\documentclass[oneside,english,american,leqno]{amsart}
\usepackage[T1]{fontenc}
\usepackage[latin9]{inputenc}
\usepackage[active]{srcltx}
\usepackage{babel}
\usepackage{array}
\usepackage{amsthm}
\usepackage{amstext}
\usepackage{amssymb}
\usepackage{graphicx}
\usepackage{esint}
\usepackage[unicode=true,
 bookmarks=true,bookmarksnumbered=false,bookmarksopen=false,
 breaklinks=false,pdfborder={0 0 0},backref=false,colorlinks=false]
 {hyperref}
\hypersetup{pdftitle={On Multiple Frequency Power Density Measurements},
 pdfauthor={Giovanni S. Alberti}}
\usepackage{breakurl}

\makeatletter

\providecommand{\tabularnewline}{\\}

\theoremstyle{plain}
\newtheorem{thm}{\protect\theoremname}[section]
  \theoremstyle{definition}
  \newtheorem{defn}[thm]{\protect\definitionname}
  \theoremstyle{definition}
  \newtheorem{problem}[thm]{\protect\problemname}
  \theoremstyle{plain}
  \newtheorem{prop}[thm]{\protect\propositionname}
  \theoremstyle{plain}
  \newtheorem{cor}[thm]{\protect\corollaryname}
  \theoremstyle{plain}
  \newtheorem{lem}[thm]{\protect\lemmaname}
  \theoremstyle{remark}
  \newtheorem{rem}[thm]{\protect\remarkname}
  \theoremstyle{definition}
  \newtheorem{example}[thm]{\protect\examplename}

\usepackage{geometry} 

\theoremstyle{definition}
\newtheorem*{defcomplete}{Definition \ref{def:complete sets}}
\newtheorem*{defproper}{Definition \ref{def:proper sets}}
\theoremstyle{plain}
\newtheorem*{thmformulaaq}{Theorem \ref{thm:formulaaq}}
\newtheorem*{thmformulaq}{Theorem \ref{thm:formulaq}}

\@ifundefined{showcaptionsetup}{}{%
 \PassOptionsToPackage{caption=false}{subfig}}
\usepackage{subfig}
\makeatother

  \addto\captionsamerican{\renewcommand{\corollaryname}{Corollary}}
  \addto\captionsamerican{\renewcommand{\definitionname}{Definition}}
  \addto\captionsamerican{\renewcommand{\examplename}{Example}}
  \addto\captionsamerican{\renewcommand{\lemmaname}{Lemma}}
  \addto\captionsamerican{\renewcommand{\problemname}{Problem}}
  \addto\captionsamerican{\renewcommand{\propositionname}{Proposition}}
  \addto\captionsamerican{\renewcommand{\remarkname}{Remark}}
  \addto\captionsamerican{\renewcommand{\theoremname}{Theorem}}
  \addto\captionsenglish{\renewcommand{\corollaryname}{Corollary}}
  \addto\captionsenglish{\renewcommand{\definitionname}{Definition}}
  \addto\captionsenglish{\renewcommand{\examplename}{Example}}
  \addto\captionsenglish{\renewcommand{\lemmaname}{Lemma}}
  \addto\captionsenglish{\renewcommand{\problemname}{Problem}}
  \addto\captionsenglish{\renewcommand{\propositionname}{Proposition}}
  \addto\captionsenglish{\renewcommand{\remarkname}{Remark}}
  \addto\captionsenglish{\renewcommand{\theoremname}{Theorem}}
  \providecommand{\corollaryname}{Corollary}
  \providecommand{\definitionname}{Definition}
  \providecommand{\examplename}{Example}
  \providecommand{\lemmaname}{Lemma}
  \providecommand{\problemname}{Problem}
  \providecommand{\propositionname}{Proposition}
  \providecommand{\remarkname}{Remark}
\providecommand{\theoremname}{Theorem}

\begin{document}
\global\long\def\R{\mathbb{R}}

\global\long\def\N{\mathbb{N}}

\global\long\def\C{\mathbb{C}}

\global\long\def\Cl{\mathcal{C}}

\global\long\def\tr{\mathrm{tr}}

\global\long\def\Se{\mathcal{S}}

\global\long\def\P{L_{\beta_{1},\beta_{2}}^{\infty}(\Omega)}

\global\long\def\V{H_{0}^{1}(\Omega;\C)}

\global\long\def\Vp{H^{-1}(\Omega;\C)}

\global\long\def\Vr{H_{0}^{1}(\Omega)}

\global\long\def\Honer{H^{1}(\Omega)}

\global\long\def\Hone{H^{1}(\Omega;\C)}

\global\long\def\Hhalf{H^{1/2}(\partial\Omega;\C)}

\global\long\def\Hhalfr{H^{1/2}(\partial\Omega)}

\global\long\def\phi{\varphi}

\global\long\def\epsilon{\varepsilon}

\global\long\def\div{{\rm div}}

\global\long\def\ld{L^{2}(\Omega;\C)}

\global\long\def\ldr{L^{2}(\Omega)}

\global\long\def\linf{L^{\infty}(\Omega;\C)}

\global\long\def\linfr{L^{\infty}(\Omega)}

\global\long\def\Conealp{\mathcal{C}^{1,\alpha}(\overline{\Omega};\C)}

\global\long\def\Conealpr{\mathcal{C}^{1,\alpha}(\overline{\Omega})}

\global\long\def\Calpr{\mathcal{C}^{0,\alpha}(\overline{\Omega};\R^{d\times d})}

\global\long\def\Calprsca{\mathcal{C}^{0,\alpha}(\overline{\Omega})}

\global\long\def\Czeroonescal{\mathcal{C}^{0,1}(\overline{\Omega})}

\global\long\def\Calpvect{\mathcal{C}^{0,\alpha}(\overline{\Omega};\R^{d})}

\global\long\def\Cone{\Cl^{1}(\overline{\Omega};\C)}

\global\long\def\Coner{\Cl^{1}(\overline{\Omega})}

\global\long\def\Kad{K_{ad}}

\global\long\def\mina{\lambda}

\global\long\def\maxa{\Lambda}

\global\long\def\minq{\beta_{1}}

\global\long\def\maxq{\beta_{2}}

\global\long\def\Czeroone{\Cl^{0,1}(\overline{\Omega};\R^{d\times d})}

\global\long\def\BMN{B\! M\! N(\Omega)}

\title[On Multiple Frequency Power Density Measurements]{On Multiple Frequency Power Density Measurements}

\author{Giovanni S. Alberti}

\address{Mathematical Institute, University of Oxford, Andrew Wiles Building,
Radcliffe Observatory Quarter, Woodstock Road, Oxford OX2 6GG, UK.}

\email{giovanni.alberti@maths.ox.ac.uk}

\date{September 24, 2013}


\begin{abstract}
We shall give a priori conditions on the illuminations $\phi_{i}$
such that the solutions to the Helmholtz equation
\[
\left\{ \begin{array}{l}
-\div(a\,\nabla u^{i})-k\, q\, u^{i}=0\qquad\text{in \ensuremath{\Omega},}\\
u^{i}=\phi_{i}\qquad\text{on \ensuremath{\partial\Omega},}
\end{array}\right.
\]
and their gradients satisfy certain non-zero and linear independence
properties inside the domain $\Omega$, provided that a finite number
of frequencies $k$ are chosen in a fixed range. These conditions
are independent of the coefficients, in contrast to the illuminations
classically constructed by means of complex geometric optics solutions.
This theory finds applications in several hybrid problems, where unknown
parameters have to be imaged from internal power density measurements.
As an example, we discuss the microwave imaging by ultrasound deformation
technique, for which we prove new reconstruction formulae.
\end{abstract}
\maketitle

\section{Introduction}

Let $d=2$ or $d=3$ be the dimension of the ambient space, $\Omega\subseteq\R^{d}$
be a smooth bounded domain and consider the Helmholtz equation 
\begin{equation}
\left\{ \begin{array}{l}
-\div(a\,\nabla u_{k}^{\phi})-k\, q\, u_{k}^{\phi}=0\qquad\text{in \ensuremath{\Omega},}\\
u_{k}^{\phi}=\phi\qquad\text{on \ensuremath{\partial\Omega},}
\end{array}\right.\label{eq:physical_modeling}
\end{equation}
where $a$ is a real and symmetric tensor satisfying the ellipticity
condition and $q>0$. Let $K_{\text{min}}<K_{\text{max}}$ be the
bounds for the possible values of $k$ and set $\Kad=[K_{\text{min}},K_{\text{max}}]$.
Problem (\ref{eq:physical_modeling}) is well-posed provided that
$k\notin\Sigma,$ the set of the Dirichlet eigenvalues.

We want to find suitable illuminations $\phi_{i}$ such that the corresponding
solutions to (\ref{eq:physical_modeling}) and their gradients satisfy
certain non-zero properties inside the domain, e.g.
\begin{equation}
\left|u_{k}^{\phi_{1}}(x)\right|>0,\qquad\left|\nabla u_{k}^{\phi_{2}}(x)\right|>0,\label{eq:intro_1}
\end{equation}
and certain linear independence properties, e.g.
\begin{equation}
\det\begin{bmatrix}\nabla u_{k}^{\phi_{2}} & \cdots & \nabla u_{k}^{\phi_{d+1}}\end{bmatrix}(x)>0,\qquad\det\begin{bmatrix}u_{k}^{\phi_{1}} & \cdots & u_{k}^{\phi_{d+1}}\\
\nabla u_{k}^{\phi_{1}} & \cdots & \nabla u_{k}^{\phi_{d+1}}
\end{bmatrix}(x)>0.\label{eq:intro_2}
\end{equation}
As discussed more precisely below, these conditions are motivated
by the reconstruction algorithms of several hybrid imaging techniques.
The problem is usually set for a fixed frequency $k\in\Kad\setminus\Sigma$.

The classical way to tackle this problem is by means of the so called
complex geometric optics solutions \cite{sylv1987,bal2011quantitative}.
These are particular high oscillatory solutions of the Helmholtz equation
and can be used to determine suitable illuminations \cite{bal2012inversediffusion,bal2012_review}.
However, they can only be constructed when the parameters are sufficiently
smooth. Furthermore, and most importantly, their construction depends
on the coefficients $a$ and $q$. Thus, this cannot be considered
a completely satisfactory answer to this issue: in inverse problems
these are usually unknown. In the case $k=0$, the absence of critical
points is studied in \cite{alessandrini1986,alessandrinimagnanini1994}.

The main contribution of this work is a very different approach to
this problem by means of a multi-frequency approach. We shall give
a priori, i.e. independent of $a$ and $q$,  conditions on the illuminations
whose corresponding solutions satisfy the required properties, provided
a finite number of frequencies are used in the range $\Kad$. More
precisely, there exist illuminations $\phi_{i}$ and a finite number
of frequencies $K\subseteq\Kad\setminus\Sigma$ such that (\ref{eq:intro_1})
and (\ref{eq:intro_2}) are satisfied for every $x\in\Omega$ and
for some $k\in K$ depending on $x$. For example, we shall show that
if $\Omega$ is convex then the choice $\{1,x_{1},x_{2}\}$ is sufficient
in two dimensions, and the same is true for $\{1,x_{1},x_{2},x_{3}\}$
in three dimensions, provided that $a$ is close to a constant. The
main idea behind this method is simple: if the illuminations are suitably
chosen then the zero level sets of functionals depending on $u_{k}^{\phi}$
and $\nabla u_{k}^{\phi}$ \emph{move} when the frequency changes.
The proof is based on the analyticity of the map $k\mapsto u_{k}^{\phi}$.
Note that an infinite number of frequencies and analyticity properties
have been used in similar contexts \cite{acosta2012multi,bao2010multi},
but here only a finite number of $k$ is needed. Simplified versions
of the main results can be found in Subsection \ref{sub:On-the-critical}.

The theory presented in this paper finds its applications in several
hybrid imaging techniques. By combining measurements coming from two
different modalities it is possible to obtain high-resolution and
high-contrast images. For a review of the state of the art in hybrid
techniques, the reader is referred to the recent works \cite{ammari2008introduction,ammari2011expansion,widlak2012hybrid,kuchment2011_review,bal2012_review}.
Generally, a hybrid problem involves two steps. First, internal energies
are measured inside the domain and, second, from their knowledge the
unknown parameters have to be reconstructed.

Many hybrid problems are governed by the Helmholtz equation (\ref{eq:physical_modeling})
or by one of its variants (complex coefficients, different type of
boundary conditions, for which a multi-frequency approach can be carried
out in the same way), e.g. microwave imaging by ultrasound deformation
\cite{triki2010,cap2011}, quantitative thermo-acoustic \cite{bal2011quantitative,ammari2012quantitative},
transient elastography and magnetic resonance elastography \cite{ammari2008elasticity,2010lee,bal2011reconstruction}
(for which the Helmholtz equation is used as a one-dimensional approximation).
The internal energies are always linear or quadratic functionals of
$u_{k}^{\phi}$ and of $\nabla u_{k}^{\phi}$ and the parameters $a$
and $q$ have to be imaged. From the reconstruction procedures discussed
in these papers, it turns out that some or all the conditions (\ref{eq:intro_1})
and (\ref{eq:intro_2}) are necessary. Thus, being able to determine
suitable illuminations independently of the unknown parameters is
fundamental, and these can be given by the multi-frequency approach
developed here. It is worth observing that the nature of this approach,
for which the frequency $k$ depends on the points in the domain,
is better suited in cases where the internal data are collected locally,
as for instance in microwave imaging by ultrasound deformation. On
the other hand, it is more restrictive when the internal data are
measured in the whole domain, as in thermo-acoustic imaging. Namely,
we will have redundant measurements in some parts of the domain, where
(\ref{eq:intro_1}) and (\ref{eq:intro_2}) are satisfied for two
or more frequencies.

As an example, we apply the theory to one of these hybrid problems,
microwave imaging by ultrasound deformation, which was introduced
in \cite{cap2011}. The internal energies have the form
\[
E=a\left|\nabla u_{k}^{\phi}\right|^{2},\qquad e=q\left(u_{k}^{\phi}\right)^{2},
\]
and the electromagnetic parameters $a$ and $q$ have to be reconstructed.
We provide reconstruction formulae for $a/q$ and $q$, which are
applicable if (\ref{eq:intro_1}) and a weaker version of (\ref{eq:intro_2})
are satisfied for a suitable set of illuminations and a finite number
of frequencies in the microwave regime.

We believe that the multi-frequency approach developed in this work
can be used in other contexts where conditions similar to (\ref{eq:intro_1})
and (\ref{eq:intro_2}) naturally arise \cite{bal2011reconstruction}.
In many situations one looks for solutions of the Helmholtz equation
satisfying certain properties: complex geometric optics is the only
tool that has been used so far \cite{bal2012_review} and presents
the difficulties described above. The proof of the results uses regularity
results for the solutions, the analyticity of $k\mapsto u_{k}^{\phi}$,
the possibility of using several frequencies in a fixed range and
the fact that for a certain value of the frequency (e.g. for $k=0$)
the required conditions are satisfied. Therefore, as long as these
requirements can be fulfilled, the multi-frequency approach can work
also with other governing equations. The Helmholtz equation is not
the only context where this theory may be applicable.

This paper is organized as follows. In Section \ref{sec:Main-Results}
we precisely describe the setting and state the main theoretical results.
Then, we apply these to a particular hybrid problem and provide reconstruction
formulae. In Section \ref{sec:Helmholtz_Equation} we consider the
mathematical aspects of the Helmholtz equation with complex $k$:
we review well-posedness and regularity results and show the analyticity
of the map $k\mapsto u_{k}^{\phi}$. In Section \ref{sec:Analyticity-Techniques}
the multi-frequency approach is discussed and used to prove the main
theorems. Section \ref{sec:Microwave-Imaging-by} is devoted to the
study of the example and to the proof of the reconstruction formulae.

Let us clarify some notation.  The space of infinitely differentiable
functions compactly supported in $\Omega$ will be denoted by $\mathcal{D}(\Omega)$.
We use the classical notation for function spaces and, unless otherwise
stated, we always consider real-valued function spaces. Whenever complex-valued
functions have to be taken, we add the letter $\C$, as for instance
in $\Hone$. Let $\gamma$ denote the trace operator from $\Hone$
to $\Hhalf$. In a Hilbert space $H$, the scalar product will be
denoted by $(\:,\,)_{H}$. We shall use the notation $\theta_{u,v}$
for the angle between two non-zero vectors $u,v$ in a Hilbert space,
i.e. $\theta_{u,v}=\arccos\frac{(u,v)}{\left\Vert u\right\Vert \left\Vert v\right\Vert }$.We
write $B(x,r)$ for the ball centered in $x\in\R^{d}$ of radius $r>0$.

\section{\label{sec:Main-Results}Main Results}

Let $d=2$ or $3$ be the dimension of the ambient space, $\Omega'\subseteq\Omega\subseteq\R^{d}$
be two $\Cl^{1,\alpha}$ bounded domains for some $\alpha\in(0,1)$.
Let $a$ be a real, symmetric tensor satisfying the ellipticity condition
\begin{equation}
\mina\left|\xi\right|^{2}\le\xi\cdot a\xi\le\maxa\left|\xi\right|^{2},\qquad\xi\in\R^{d},\label{eq:assumption_a}
\end{equation}
for some $\mina,\maxa>0$. We also suppose that $a$ is of class $\Calpr$.
Let $q\in\linfr$ satisfy
\begin{equation}
\minq\le q\le\maxq\text{ almost everywhere in }\Omega,\label{eq:assumption_q}
\end{equation}
for some $\minq,\maxq>0$. Let $\Sigma$ denote the set of the Dirichlet
eigenvalues of the problem 
\begin{equation}
\left\{ \begin{array}{l}
-\div(a\,\nabla u_{k}^{\phi})-k\, q\, u_{k}^{\phi}=0\qquad\text{in \ensuremath{\Omega},}\\
u_{k}^{\phi}=\phi\qquad\text{on \ensuremath{\partial\Omega},}
\end{array}\right.\tag{\ref{eq:physical_modeling}}\label{eq:physical_modeling-1}
\end{equation}
which is a countable set of positive numbers going to infinity. For
any $k\in\C\setminus\Sigma$ problem (\ref{eq:physical_modeling-1})
is well-posed and $u_{k}^{\phi}\in\Cone$ (see Section 3). Let $K_{\text{min}}<K_{\text{max}}$
be the bounds for the possible wavenumbers $k$ and denote $\Kad=[K_{\text{min}},K_{\text{max}}]$.
More generally, we can assume that $\Kad$ is a non-empty continuous
curve in $\C$, containing the admissible values for the frequencies
$k$ and depending on the particular problem under consideration.

In this section we summarize the main results of this paper. In Subsection
\ref{sub:On-the-critical}, the theoretical results regarding non-zero
properties of solutions of the Helmholtz are discussed. In Subsection
\ref{sub:Applications-to-Microwave} we apply the general theory to
a particular case.

\subsection{\label{sub:On-the-critical}A Multiple Frequency Approach to the
Boundary Control of the Helmholtz Equation}

Given finite subsets $K\subseteq\Kad\setminus\Sigma$ and $I\subseteq\gamma\left(\Conealpr\right)$,
the Cartesian product $K\times I$ is called a \emph{set of measurements.}

Given a set of measurements $K\times\{\phi_{i}\}$, we consider the
unique solution $u_{k}^{i}\in\Coner$ to the problem
\begin{equation}
\left\{ \begin{array}{l}
-\div(a\,\nabla u_{k}^{i})-k\, q\, u_{k}^{i}=0\qquad\text{in \ensuremath{\Omega},}\\
u_{k}^{i}=\phi_{i}\qquad\text{on \ensuremath{\partial\Omega}.}
\end{array}\right.\label{eq:hel-2}
\end{equation}
The introduction of the following class of sets of measurements is
motivated by several hybrid techniques \cite{triki2010,bal2011quantitative,cap2011,ammari2012quantitative}.
In this work we will not apply the full generality of this concept
(see Definition \ref{def:proper sets}).
\begin{defn}
\label{def:complete sets}Let $p,r,s>0$. A set of measurements $K\times\left\{ \phi_{i}:i=1,\dots,d+1\right\} $
is \emph{complete} in $\Omega'$ if for every $x\in\Omega'$ there
exists $\bar{k}=\bar{k}(x)\in K$ such that:
\begin{equation}
\left|u_{\bar{k}}^{1}(x)\right|\ge p,\tag{CSM 1}\label{eq:complete_i)}
\end{equation}
\begin{equation}
\left|\det\begin{bmatrix}\nabla u_{\bar{k}}^{2} & \cdots & \nabla u_{\bar{k}}^{d+1}\end{bmatrix}(x)\right|\ge r,\tag{CSM 2}\label{eq:complete_ii)}
\end{equation}
\begin{equation}
\bigl|\det\begin{bmatrix}u_{\bar{k}}^{1} & \cdots & u_{\bar{k}}^{d+1}\\
\nabla u_{\bar{k}}^{1} & \cdots & \nabla u_{\bar{k}}^{d+1}
\end{bmatrix}(x)\bigr|\ge s.\tag{CSM 3}\label{eq:complete_iii)}
\end{equation}

\end{defn}
The conditions (\ref{eq:complete_i)}) and (\ref{eq:complete_ii)})
guarantee that meaningful power density measurements exist in every
point. Further, (\ref{eq:complete_ii)}) and (\ref{eq:complete_iii)})
imply that there are $d$ and $d+1$ independent measurements everywhere,
respectively.

The existence of complete sets of measurements is a non trivial problem
for general $a$ and $q$. In the references cited above, some or
all these conditions are shown by using complex geometric optics solutions
\cite{sylv1987}. In view of Proposition 3.3 in \cite{bal2010inverse},
suitable illuminations can be constructed with a fixed frequency $k_{0}$,
provided that $a$ and $q$ are smooth enough. However, their construction
depends on the knowledge of $a$ and $q$, that in the inverse problems
we have in mind are unknown.

In order to tackle this issue, in Section \ref{sec:Analyticity-Techniques}
we discuss a multiple frequency approach to construct complete sets
of measurements. The conditions on the illuminations are independent
of the coefficients. Here we provide a simplified version of the main
results therein discussed.

The next theorem discusses the construction of complete sets of measurements
in dimension $d=2$. It is a particular case of Theorem \ref{thm:complete_2d}.
\begin{thm}
\label{thm:main_2d}Suppose $d=2$, $a\in\mathcal{C}^{0,1}(\overline{\Omega})$
and that $\Omega$ is convex and $\mathcal{C}^{2}$. Then we can choose
a finite $K\subseteq\Kad\setminus\Sigma$ such that 
\[
K\times\{1,x_{1},x_{2}\}
\]
is a complete set of measurements in $\Omega$.
\end{thm}
The next theorem deals with the construction of complete sets of measurements
in dimension $d=3$. It is a particular case of Theorem \ref{thm:complete_3D}.
\begin{thm}
\label{thm:main_3d}Suppose $d=3$. Let $a_{0}$ be a constant, symmetric
and positive definite matrix. There exists $\delta>0$ such that for
any $s\in\Calpr$ with $\left\Vert s\right\Vert _{\Calpr}\le\delta$
we can choose a finite $K\subseteq\Kad\setminus\Sigma$ such that
\[
K\times\{1,x_{1},x_{2},x_{3}\}
\]
is a complete set of measurements in $\Omega$ for $a=a_{0}+s$.
\end{thm}
The proofs of these statements can be found in Section \ref{sec:Analyticity-Techniques},
which also contains more general versions of the previous results.
In particular, we shall see that in two dimensions the illuminations
$1,x_{1}$ and $x_{2}$ represent a simple selection from a much wider
class and the domain does not have to be convex. Moreover, we describe
how to choose the finite set of frequencies $K$.

\subsection{\label{sub:Applications-to-Microwave}Applications to Microwave Imaging
by Ultrasound Deformation}

We now discuss an example to justify the introduction of complete
sets. In fact, for our purposes, a wider class of sets of measurements
is sufficient.

We consider the hybrid problem arising from the combination of microwaves
and ultrasounds which was introduced in \cite{cap2011}. Full details
and the proofs of the results are given in Section \ref{sec:Microwave-Imaging-by}.
In addition to the previous assumptions, we suppose that $a$ is scalar-valued.
In microwave imaging, $a$ is the inverse of the magnetic permeability,
$q$ is the electric permittivity and $ $$\Kad=[K_{\text{min}},K_{\text{max}}]$,
with $K_{\text{min}}>0$, are the admissible frequencies in the microwave
regime.

Given a set of measurements $K\times\left\{ \phi_{i}\right\} $ we
consider power density measurements of the form
\begin{equation}
e_{k}^{ij}=q\, u_{k}^{i}\, u_{k}^{j},\qquad E_{k}^{ij}=a\,\nabla u_{k}^{i}\cdot\nabla u_{k}^{j},\label{eq:internal energies-1}
\end{equation}
where $u_{k}^{i}$ is given by (\ref{eq:hel-2}). For simplicity,
we denote $e_{k}=(e_{k}^{ij})_{ij}$ and similarly for $E$. These
internal energies have to be considered as known functions in $\Omega'$.
\begin{problem}
\label{prob:Choose-a-suitable}Choose a suitable set of measurements
\foreignlanguage{english}{$K\times\left\{ \phi_{i}\right\} $} and
find $a$ and $q$ in $\Omega'$ from the knowledge of $e_{k}^{ij}$
and $E_{k}^{ij}$ in $\Omega'$.
\end{problem}
Problem \ref{prob:Choose-a-suitable} is solved via two reconstruction
formulae for $a/q$ and $q$, respectively, which we shall now describe.
Their applicability is guaranteed if \foreignlanguage{english}{$K\times\left\{ \phi_{i}\right\} $}
is a proper set of measurements in $\Omega'$.
\begin{defn}
\label{def:proper sets}Let $p$ and $r$ be two positive constants.
A set of measurements $K\times\left\{ \phi_{i}:i=1,2,3\right\} $
is  \emph{proper} in $\Omega'$ if for every $x\in\Omega'$ there
exists $\bar{k}=\bar{k}(x)\in K$ such that:
\begin{equation}
\left|u_{\bar{k}}^{1}(x)\right|\ge p,\tag{PSM 1}\label{eq:proper_i)}
\end{equation}
\begin{equation}
\bigl|\nabla u_{\bar{k}}^{2}(x)\bigr|\bigl|\nabla u_{\bar{k}}^{3}(x)\bigr|\left|\sin\theta_{\nabla u_{\bar{k}}^{2},\nabla u_{\bar{k}}^{3}}(x)\right|\ge r.\tag{PSM 2}\label{eq:proper_ii)}
\end{equation}
The collection of all proper sets of measurements in $\Omega'$ with
constants $p$ and $r$ will be denoted by $\mathcal{P}(\Omega';p,r)$.
\end{defn}
Clearly, the conditions characterizing a proper set are weaker than
the conditions of a complete set: the main difference relies in the
requirement of only two independent measurements in any dimension.
Thus, the construction of proper sets can be easily achieved by using
the Theorems of the previous subsection (in dimension three the illumination
$x_{3}$ is not needed).

The next statement gives an exact formula for $a/q$. The formula
was first derived in \cite{cap2011} for the two-dimensional case,
and we have extended it to any dimension. We use the following notation
for a matrix-valued function $M$ of size $N$
\[
\tr(M)=\sum_{i=1}^{N}M_{ii},\qquad\left|\nabla M\right|_{2}^{2}=\sum_{i,j=1}^{N}\left|\nabla M_{ij}\right|^{2}.
\]

\begin{thm}
\label{thm:formulaaq} Take $p,r>0$ and $K\times\left\{ \phi_{i}\right\} \in\mathcal{P}(\Omega';p,r)$.
Suppose that
\begin{equation}
\left|E_{k}^{ii}(x)\right|,\left|e_{k}^{ii}(x)\right|\le b,\qquad x\in\Omega',\label{eq:hp_with_M}
\end{equation}
for some $b>0$. Take $x\in\Omega'$ and $\bar{k}$ as in Definition
\ref{def:proper sets}. Then there exists $C=C(\mina,\minq,b)>0$
such that 
\begin{equation}
\frac{\tr(e_{\bar{k}})\,\tr(E_{\bar{k}})-\tr(e_{\bar{k}}E_{\bar{k}})}{\tr(e_{\bar{k}})^{2}}(x)\ge C\, p^{2}r^{4},\label{eq:positivity of traces}
\end{equation}
and $a/q$ is given in terms of the data by
\begin{equation}
\left|\nabla(e_{\bar{k}}/\tr(e_{\bar{k}}))\right|_{2}^{2}\,\frac{a}{q}=2\,\frac{\tr(e_{\bar{k}})\,\tr(E_{\bar{k}})-\tr(e_{\bar{k}}E_{\bar{k}})}{\tr(e_{\bar{k}})^{2}}\quad\text{in }x.\label{eq:formula for a/q}
\end{equation}

\end{thm}
We now give an exact formula for $q$. If compared to the one described
in \cite{cap2011}, this one is valid in any dimension and does not
require the set of measurements to be complete. It only requires that
\[
\tr(e):=\sum_{i,k}e_{k}^{ii}\ge c>0\quad\text{in }\Omega',
\]
which holds provided that the set of measurements satisfies (\ref{eq:proper_i)}).
\begin{thm}
\label{thm:formulaq}Let \foreignlanguage{english}{\textup{$K\times\left\{ \phi_{i}\right\} $}}
be a proper set of measurements. Suppose $q\in\Honer$. Then $\log q$
is the unique solution to the problem
\[
\left\{ \begin{array}{l}
-\div\left(G\,\tr(e)\,\nabla u\right)=-\div\left(G\,\nabla\left(\tr(e)\right)\right)+2\sum_{k,i}\left(E_{k}^{ii}-ke_{k}^{ii}\right)\quad\text{in }\Omega',\\
u=\log q_{|\partial\Omega'}\qquad\text{on \ensuremath{\partial\Omega}}'.
\end{array}\right.
\]

\end{thm}
The previous result allows to reconstruct $q$ from the knowledge
of $G=a/q$. Also, $q$ is supposed to be known on the boundary of
$\Omega'$, which is a reasonable assumption since we are mainly interested
in detecting inclusions inside the domain. At this point, an obvious
formula for $a$ is $a=Gq$.

The two exact formulae provided by the previous Theorems use the derivatives
of the data. As a consequence, in presence of noise, they may not
give good quality images. Thus, an optimal control approach would
be an efficient way to find better approximations of the coefficients,
starting from the good guess given by the formulae described here.
Details can be found in \cite{cap2011}.

\section{\label{sec:Helmholtz_Equation}The Helmholtz Equation}

In this section the Helmholtz equation (\ref{eq:physical_modeling})
is discussed, where $a\in\Calpr$ and $q\in\linfr$ are as in Section
\ref{sec:Main-Results} and satisfy (\ref{eq:assumption_a}) and (\ref{eq:assumption_q}).
First, we show that the problem is well-posed in the general situation
of complex $k$. Then, regularity results are discussed. In order
to do this, we use the classical results on elliptic problems. The
aim of this analysis is the proof of the analyticity of the map $k\mapsto u_{k}^{\phi}$,
which is given in Subsection \ref{sub:Analyticity-Properties}.

\subsection{Well-Posedness and Regularity}

First of all, we study existence, uniqueness and stability with homogeneous
Dirichlet boundary conditions. The space $\Vp$ denotes the continuous
antidual of $\V$.
\begin{prop}
\label{pro:well-posed}There exists $\Sigma=\left\{ \lambda_{i}:i\in\N\right\} \subseteq\R_{+}$
with \textup{$\lambda_{i}\to+\infty$} such that for $k\in\C\setminus\Sigma$
and $f\in\Vp$ the equation 
\begin{equation}
-\div(a\nabla u)-k\, q\, u=f\label{eq:Bk}
\end{equation}
has a unique solution $u\in\V$ satisfying $\left\Vert u\right\Vert _{H_{0}^{1}}\le C\left\Vert f\right\Vert _{H^{-1}}$
for some $C>0$ independent of $f$. Moreover, for fixed $f\in\Vp$,
the map
\[
k\in\C\setminus\Sigma\longmapsto u\in\V
\]
is analytic.\end{prop}
\begin{proof}
We give a sketch of the proof. Consider $L=-\div(a\,\nabla\,\cdot\,)\colon\V\to\Vp$
and $M_{q}\colon\ld\to\ld$ defined by $f\mapsto q\, f$. By Lax-Milgram
Theorem and (\ref{eq:assumption_a}) the operator $L$ is invertible
and we can consider $S:=L^{-1}M_{q}\colon\V\to\V$, which by Kondrachov
Compactness Theorem and (\ref{eq:assumption_q}) is compact, self-adjoint
and positive. Hence $S$ has a countable set of eigenvalues $\left\{ \eta_{i}>0:i\in\N\right\} $,
with $\eta_{i}\to0$. Define $\Sigma=\left\{ \lambda_{i}=1/\eta_{i}:i\in\N\right\} $
and take $k\in\C\setminus\Sigma$. Since (\ref{eq:Bk}) is equivalent
to $\left(I-k\, S\right)u=L^{-1}f$, the first part follows. Finally,
the analyticity of $k\mapsto u$ is a consequence of the so called
Analytic Fredholm Theorem (see \cite[Theorem 8.92]{renardy2004introduction}).
\end{proof}
As a consequence, we immediately get the following result regarding
the Dirichlet boundary value problem for the Helmholtz equation.
\begin{cor}
\label{cor:bvp}Take $k\in\C\setminus\Sigma$. Then there exists $C=C(\Omega,a,q,k)>0$
such that for every $f\in\Vp$ and $\phi\in\Hhalf$ the problem
\begin{equation}
\left\{ \begin{array}{l}
-\div(a\,\nabla u)-k\, q\, u=f\qquad\text{in \ensuremath{\Omega},}\\
u=\phi\qquad\text{\text{on \ensuremath{\partial\Omega},}}
\end{array}\right.\label{eq:Helbound}
\end{equation}
has a unique solution $u\in\Hone$ satisfying
\begin{equation}
\left\Vert u\right\Vert _{\Hone}\le C\left(\left\Vert \phi\right\Vert _{\Hhalf}+\left\Vert f\right\Vert _{\Vp}\right).\label{eq:helboundstab}
\end{equation}
Moreover, for fixed $f\in\Vp$ and $\phi\in\Hhalf$, the map
\[
k\in\C\setminus\Sigma\longmapsto u\in\Hone
\]
is analytic.
\end{cor}
Standard elliptic regularity theory allows us to study the regularity
of the solution $u\in\Hone$ to (\ref{eq:Helbound}).
\begin{prop}
\label{pro:holder}Take $k\in\C\setminus\Sigma$, $f\in\linf$, $v\in\Conealp$
and write $\phi=\gamma(v)$. Let $u\in\Hone$ be the unique solution
to (\ref{eq:Helbound}). Then $u\in\Cone$ and there exists $C=C(\Omega,a,q,k)>0$
such that
\begin{equation}
\left\Vert u\right\Vert _{\Cone}\le C\left(\left\Vert v\right\Vert _{\Conealp}+\left\Vert f\right\Vert _{\infty}\right).\label{eq:0}
\end{equation}
\end{prop}
\begin{proof}
We decompose $u,\, f,\,\phi$ and $k$ into real and imaginary parts
by writing $u=u_{R}+i\, u_{I}$, $f=f_{R}+i\, f_{I}$, $v=v_{R}+i\, v_{I}$
and $k=k_{R}+i\, k_{I}$. By testing the first equation of (\ref{eq:Helbound})
against any $w\in H_{0}^{1}(\Omega;\R)$ and by taking real and imaginary
parts we obtain
\begin{align*}
\begin{cases}
-\div(a\nabla u_{R})-k_{R}\, q\, u_{R}=-k_{I}\, q\, u_{I}+f_{R},\quad & \gamma(u_{R})=\phi_{R},\\
-\div(a\nabla u_{I})-k_{R}\, q\, u_{I}=k_{I}\, q\, u_{R}+f_{I}, & \gamma(u_{I})=\phi_{I}.
\end{cases}
\end{align*}
In the rest of the proof we will consider constants $c=c(\Omega,a,q,k)$.
From usual elliptic regularity theory (e.g. (8.38) in \cite{gilbarg2001elliptic})
there holds
\[
\begin{split}\left\Vert u_{R}\right\Vert _{\infty} & \le c\left(\left\Vert u\right\Vert _{2}+\left\Vert v\right\Vert _{\Conealp}+\left\Vert f\right\Vert _{\infty}\right)\end{split}
\]
By arguing in the same way with $u_{I}$ we infer that
\[
\left\Vert u\right\Vert _{\infty}\le c\left(\left\Vert u\right\Vert _{2}+\left\Vert v\right\Vert _{\Conealp}+\left\Vert f\right\Vert _{\infty}\right).
\]
Therefore, by Corollary 8.35 in \cite{gilbarg2001elliptic} we get
that $u_{R},\, u_{I}\in\Coner$ and that
\[
\begin{split}\left\Vert u_{R}\right\Vert _{\Coner} & \le c\left(\left\Vert u\right\Vert _{2}+\left\Vert v\right\Vert _{\Conealp}+\left\Vert f\right\Vert _{\infty}\right)\end{split}
.
\]
A similar inequality holds for $u_{I}$ and so
\[
\left\Vert u\right\Vert _{\Cone}\le c\left(\left\Vert u\right\Vert _{2}+\left\Vert v\right\Vert _{\Conealp}+\left\Vert f\right\Vert _{\infty}\right).
\]
Finally, in view of (\ref{eq:helboundstab}) we obtain (\ref{eq:0}).
\end{proof}

\subsection{Analyticity Properties\label{sub:Analyticity-Properties}}

First, we need the following lemma concerning the composition of analytic
functions.
\begin{lem}[\cite{Whittlesey1965}]
\label{lem:composition_analytic}Let $D\subseteq\C$ be open, $X,Y$
be Banach spaces, $f_{i}\colon D\to X$ be analytic maps for $i=1,\dots,b$
and $g\colon X^{b}\to Y$ be multilinear and bounded. Then 
\[
g\circ(f_{1},\dots,f_{b})\colon D\longrightarrow Y
\]
 is analytic.
\end{lem}
In this subsection we study the dependence of the solutions $u_{k}^{\phi}$
on $k$. We come back to the original problem 
\begin{equation}
\left\{ \begin{array}{l}
-\div(a\,\nabla u_{k}^{\phi})-k\, q\, u_{k}^{\phi}=0\qquad\text{\text{in \ensuremath{\Omega},}}\\
u_{k}^{\phi}=\phi\qquad\text{on \ensuremath{\partial\Omega},}
\end{array}\right.\tag{\ref{eq:physical_modeling}}\label{eq:hel_repeated}
\end{equation}
for $k\in\C\setminus\Sigma$ and fixed $\phi\in\gamma\left(\Conealpr\right)$.
By Corollary \ref{cor:bvp} there exists a unique solution $u_{k}^{\phi}\in\Hone$,
that in view of Proposition \ref{pro:holder} is in $\Cone$. We have
already seen that $u_{k}^{\phi}$ depends analytically on the wavenumber
$k$ with respect to the norm of $\Hone$ . We now want to show this
with respect to the $\Cone$ norm. The proof follows the  argument
used by Calderón to prove that $u_{k}^{\phi}$ depends analytically
on $a$ if $k=0$ \cite{MR2046906}.
\begin{prop}
\label{pro:T_analytic}Take $\phi\in\gamma\left(\Conealpr\right)$.
Then the map
\[
T\colon\C\setminus\Sigma\longrightarrow\Cone,\quad k\longmapsto u_{k}^{\phi}
\]
is analytic, where $u_{k}^{\phi}$ is the unique solution to (\ref{eq:hel_repeated}).\end{prop}
\begin{proof}
For simplicity, we do not write the dependence on $\phi$. Fix $k_{0}\in\C\setminus\Sigma$:
we shall prove that $T$ is analytic in $k_{0}$.

For $h\in\C$ define the operator $C_{h}\colon\Cone\to\linf$, $w\mapsto h\, q\, w$
with norm $\left\Vert C_{h}\right\Vert \le\maxq\left|h\right|$. In
view of Proposition \ref{pro:holder}, there exists a bounded operator
$B\colon\linf\longrightarrow\Cone$, where $Bf$ is the unique solution
to
\[
\left\{ \begin{array}{l}
-\div(a\,\nabla u)-k\, q\, u=f\qquad\text{in \ensuremath{\Omega},}\\
u=0\qquad\text{\text{on \ensuremath{\partial\Omega}.}}
\end{array}\right.
\]
As a consequence, $BC_{h}\colon\Cone\to\Cone$ is a bounded operator
with norm $\left\Vert BC_{h}\right\Vert \le\maxq\left\Vert B\right\Vert \left|h\right|$.
Define $r=\min\left\{ 1/(\maxq\left\Vert B\right\Vert ),d(k_{0},\Sigma)\right\} $,
where $d(k_{0},\Sigma)=\inf\left\{ d(k_{0},\sigma):\sigma\in\Sigma\right\} $
denotes the distance between $k_{0}$ and $\Sigma$. Take $k\in B(k_{0},r)\subseteq\C\setminus\Sigma.$
The increment $h=k-k_{0}$ satisfies $\left|h\right|<r\le\frac{1}{\maxq\left\Vert B\right\Vert }$,
whence $\left\Vert BC_{h}\right\Vert <1$. Denoting $v=u_{k}-u_{k_{0}}$,
by definition of $T$ we have 
\[
-\div\left(a\,\nabla v\right)-k_{0}q\, v-h\, q\, v=h\, q\, u_{k_{0}}=C_{h}u_{k_{0}}
\]
which, after applying $B$ to both sides, becomes $\left(I-BC_{h}\right)v=BC_{h}u_{k_{0}}$,
where $I$ denotes the identity operator of $\Cone$. As a result,
since $\left\Vert BC_{h}\right\Vert <1$, we obtain 
\[
v=\left(I-BC_{h}\right)^{-1}BC_{h}u_{k_{0}}=\sum_{n=1}^{\infty}\left(BC_{h}\right)^{n}u_{k_{0}}.
\]
It follows that
\begin{equation}
T(k)=\sum_{n=0}^{\infty}\left(BC_{k-k_{0}}\right)^{n}T(k_{0}),\quad k\in B(k_{0},r).\label{eq:expansion_T1}
\end{equation}
Namely, the map $T$ is analytic. 
\end{proof}
For the sake of completeness, we now show that the dependence on $k$
of the internal energies (\ref{eq:internal energies-1}) is analytic.
Consider two different boundary data $\phi_{1},\phi_{2}\in\Hhalf$
and the corresponding solutions $u_{k}^{i}:=u_{k}^{\phi_{i}}$ to
(\ref{eq:hel_repeated}). For $i,j=1,2$ we define the internal energies
as \begin{align*}   e^{ij}\colon&\C\setminus\Sigma\longrightarrow L^{1}(\Omega;\C), &   E^{ij}\colon&\C\setminus\Sigma\longrightarrow L^{1}(\Omega;\C),\\    &k\longmapsto q\, u_{k}^{i}\, u_{k}^{j}, &  &k\longmapsto a\,\nabla u_{k}^{i}\cdot\nabla u_{k}^{j}.\end{align*}
\begin{thm}
\label{thm:e_E_analytic}Take $\phi_{1},\phi_{2}\in\Hhalf$. Then
the internal energies $e^{ij}$ and $E^{ij}$ are analytic functions.\end{thm}
\begin{proof}
Note that the maps \begin{align*}   g_1\colon\Hone^2&\longrightarrow L^{1}(\Omega;\C), &   g_2\colon\Hone^2&\longrightarrow L^{1}(\Omega;\C)\\    (u,v)&\longmapsto q\, u\, v &  (u,v)&\longmapsto a\,\nabla u\cdot\nabla u\end{align*}are
bilinear and bounded. Therefore the result follows by Corollary \ref{cor:bvp}
and Lemma \ref{lem:composition_analytic}.
\end{proof}
As a consequence of the unique continuation property for holomorphic
functions we obtain the following
\begin{cor}
\label{cor:e_E_analytic}Take $\phi_{1},\phi_{2}\in\Hhalf$. Let $k_{l},k\in\C\setminus\Sigma$
such that $k_{l}\to k$ and $k_{l}\neq k$ for every $l\in\N$. If
$e^{ij}$ {[}resp. $E^{ij}${]} is known in $k_{l}$ for every $l\in\N$
then $e^{ij}$ {[}resp. $E^{ij}${]} is known everywhere in $\C\setminus\Sigma$.\end{cor}
\begin{rem}
This result must be seen in view of Problem \ref{prob:Choose-a-suitable}.
The knowledge of $E^{ij}$ for infinitely many frequency in a fixed
range determines $E^{ij}$ for $k=0$, where the reconstruction process
for $a$ has been studied thoroughly \cite{cap2008,2012ammari_resolution,bal2012inversediffusion,cap2009,2012balanisotropicdiffusion,widlak2012hybrid}.
However, this has to be regarded simply as an interesting theoretical
result: analytic continuation is a very ill-posed process.
\end{rem}

\section{A Multiple Frequency Approach to the Boundary Control of the Helmholtz
Equation\label{sec:Analyticity-Techniques}}

In this section we discuss the multi-frequency approach to the issue
of the existence of complete and proper sets of measurements. The
theorems stated in Subsection \ref{sub:On-the-critical} will be a
consequence of the results presented here.

\subsection{\label{sub:Preliminary-result}Preliminaries}

Let $a\in\Calpr$ and $q\in\linfr$ be as in Section \ref{sec:Main-Results}
and satisfy (\ref{eq:assumption_a}) and (\ref{eq:assumption_q}).
Recall that $\Kad$ is the set for the possible wavenumber, namely
we are allowed to choose values $k\in\Kad\setminus\Sigma$. As in
Proposition \ref{pro:T_analytic}, we denote the unique solution to
the boundary value problem (\ref{eq:hel_repeated}) by $u_{k}^{\phi}=T^{\phi}(k)\in\Cone$
\begin{equation}
\left\{ \begin{array}{l}
-\div(a\,\nabla u_{k}^{\phi})-k\, q\, u_{k}^{\phi}=0\qquad\text{in \ensuremath{\Omega},}\\
u_{k}^{\phi}=\phi\qquad\text{on \ensuremath{\partial\Omega}.}
\end{array}\right.\tag{\ref{eq:hel_repeated}}\label{eq:hel-1}
\end{equation}

As a consequence of the analyticity of $k\mapsto u_{k}^{\phi}$ proved
in Proposition \ref{pro:T_analytic} we obtain the following 
\begin{lem}
\label{lem:preliminary_analyticity}Take $\Omega'\subseteq\Omega$,
$b\in\N^{*}$ and $\zeta\colon\Cone^{b}\to\Cl\left(\overline{\Omega'};\C\right)$
multilinear and bounded. Let $\phi_{1},\dots,\phi_{b}\in\gamma\left(\Conealpr\right)$
be such that 
\begin{equation}
\zeta\left(u_{k^{x}}^{\phi_{1}},\dots,u_{k^{x}}^{\phi_{b}}\right)(x)\neq0,\qquad x\in\overline{\Omega'},\label{eq:hp_lemma_sets}
\end{equation}
for some $k^{x}\in\C\setminus\Sigma$. Take $k_{l},k\in\C\setminus\Sigma$
with $k_{l}\to k$ and $k_{l}\neq k$. Then there exists a finite
$L\subseteq\N$ such that
\begin{equation}
\sum_{l\in L}\left|\zeta\left(u_{k_{l}}^{\phi_{1}},\dots,u_{k_{l}}^{\phi_{b}}\right)(x)\right|>0,\qquad x\in\overline{\Omega'}.\label{eq:th_lemma_sets_0}
\end{equation}
In particular, we can choose a finite $K\subseteq\Kad\setminus\Sigma$
such that 
\begin{equation}
\sum_{k\in K}\left|\zeta\left(u_{k}^{\phi_{1}},\dots,u_{k}^{\phi_{b}}\right)(x)\right|>0,\qquad x\in\overline{\Omega'}.\label{eq:th_lemma_sets}
\end{equation}
\end{lem}
\begin{rem}
For the sake of completeness, we note that the statement holds true
for $\zeta$ analytic, but our current applications do not need this
extension. We shall use this lemma only with $k^{x}=0$. However,
this more general version allows other possible applications. In particular,
the case $\left|k^{x}\right|\to\infty$ could be studied. This could
be considered in the context of high-frequency approximations of the
Helmholtz equation.
\end{rem}

\begin{rem}
If in addition to the assumptions of this lemma we suppose that the
set $\{x\in\overline{\Omega'}:$ $\zeta\left(u_{k}^{\phi_{1}},\dots,u_{k}^{\phi_{b}}\right)(x)=0\}$
is finite, then there exists $\bar{l}\in\N$ such that
\[
\left|\zeta\left(u_{k}^{\phi_{1}},\dots,u_{k}^{\phi_{b}}\right)(x)\right|+\bigl|\zeta\bigl(u_{k_{\bar{l}}}^{\phi_{1}},\dots,u_{k_{\bar{l}}}^{\phi_{b}}\bigr)(x)\bigr|>0,\qquad x\in\overline{\Omega'}.
\]
Namely, if the zeros of $\zeta\left(u_{k}^{\phi_{1}},\dots,u_{k}^{\phi_{b}}\right)$
are isolated points then only two frequencies are necessary to obtain
a non-zero quantity everywhere. We leave the proof of this extension
to the reader.\end{rem}
\begin{proof}
Take $x\in\overline{\Omega'}.$ The map $ev_{x}\colon\Cl\left(\overline{\Omega'};\C\right)\to\C,\; u\mapsto u(x)$
is linear and bounded. Hence, using Proposition \ref{pro:T_analytic}
and Lemma \ref{lem:composition_analytic}, the map
\[
z_{x}:=ev_{x}\circ\zeta\circ\left(T^{\phi_{1}},\dots,T^{\phi_{b}}\right)\colon\C\setminus\Sigma\longrightarrow\C
\]
is analytic. By (\ref{eq:hp_lemma_sets}) we have $z_{x}(k^{x})\neq0$,
and so the set $\left\{ k'\in\C\setminus\Sigma:z_{x}(k')=0\right\} $
has no accumulation points in $\C\setminus\Sigma$ by the analytic
continuation theorem. Since by assumption $k$ is an accumulation
point for the sequence $(k_{l})$, this implies that there exists
$l_{x}\in\N$ such that $z_{x}(k_{l_{x}})\neq0$. Since $\zeta\circ\left(T^{\phi_{1}},\dots,T^{\phi_{b}}\right)(k_{l_{x}})$
is continuous, we can find $r_{x}>0$ such that 
\begin{equation}
z_{y}(k_{l_{x}})\neq0,\qquad y\in B(x,r_{x})\cap\overline{\Omega'}.\label{eq:10}
\end{equation}
Since $\overline{\Omega'}=\bigcup_{x\in\overline{\Omega'}}\left(B(x,r_{x})\cap\overline{\Omega'}\right)$
there exist $x_{1},\dots,x_{N}\in\overline{\Omega'}$ satisfying 
\begin{equation}
\overline{\Omega'}=\bigcup_{i=1}^{N}\left(B(x_{i},r_{x_{i}})\cap\overline{\Omega'}\right).\label{eq:11}
\end{equation}
Defining $L=\left\{ l_{x_{i}}:i=1:,\dots,N\right\} ,$ by (\ref{eq:10})
and (\ref{eq:11}) we obtain (\ref{eq:th_lemma_sets_0}).
\end{proof}
Let us now recall the definition of complete sets of measurements.

\begin{defcomplete}\label{def:complete sets-1}Let $p,r,s>0$. A set of measurements $K\times\left\{ \phi_{i}:i=1,\dots,d+1\right\} $ is \emph{complete} in $\Omega'$ if for every $x\in\Omega'$ there exists $\bar{k}=\bar{k}(x)\in K$ such that: \begin{equation} \left|u_{\bar{k}}^{1}(x)\right|\ge p,\tag{CSM 1} \end{equation} \begin{equation} \left|\det\begin{bmatrix}\nabla u_{\bar{k}}^{2} & \cdots & \nabla u_{\bar{k}}^{d+1}\end{bmatrix}(x)\right|\ge r,\tag{CSM 2} \end{equation} \begin{equation} \bigl|\det\begin{bmatrix}u_{\bar{k}}^{1} & \cdots & u_{\bar{k}}^{d+1}\\ \nabla u_{\bar{k}}^{1} & \cdots & \nabla u_{\bar{k}}^{d+1} \end{bmatrix}(x)\bigr|\ge s.\tag{CSM 3} \end{equation}
\end{defcomplete}

In order to satisfy these conditions, the main idea is to trace back
to the case $k=0$, where things are simpler. For instance, consider
condition (\ref{eq:complete_i)}). With $k=0$, by the Strong Maximum
Principle this is trivially satisfied provided that the boundary condition
is strictly positive or negative. It remains to show that this good
property transfers to any range of frequencies, and this is achieved
with Lemma \ref{lem:preliminary_analyticity}. In summary, the strategy
to study these conditions can be outlined in the following three steps:
\begin{enumerate}
\item choose a suitable $\zeta$ as above such that the condition we want
to prove is equivalent to (\ref{eq:th_lemma_sets});
\item prove the assumption (\ref{eq:hp_lemma_sets}) for $k^{x}=0$, which
is in general easier than (\ref{eq:th_lemma_sets}) since $k=0$;
\item apply Lemma \ref{lem:preliminary_analyticity} to deduce the result.
\end{enumerate}
We will deal with the issue of the construction of complete sets of
measurements in two different situations, depending on the dimension.
As we shall see, condition (\ref{eq:complete_i)}) will be a consequence
of the Maximum Principle, that holds in any dimension and for any
$a$. However, the study of conditions (\ref{eq:complete_ii)}) and
(\ref{eq:complete_iii)}) in the case $k=0$ depends on the dimension.

\subsection{Complete Sets: Two-Dimensional Case\label{sub:Two-dimensional-case}}

Throughout this subsection we assume $d=2$. First, we consider conditions
(\ref{eq:complete_i)}), (\ref{eq:complete_ii)}) and (\ref{eq:complete_iii)})
for $k=0$. As far as (\ref{eq:complete_i)}) is concerned, the Maximum
Principle will be the main tool, and no further investigation is required.

Let us now focus on (\ref{eq:complete_ii)}) for $k=0$, which reads
\[
\det\begin{bmatrix}\nabla u_{0}^{2} & \nabla u_{0}^{3}\end{bmatrix}\ge r
\]
for some $r>0$. This problem has been studied by Bauman et al. \cite{bauman2001univalent}
in the context of univalent mappings. We introduce the following class
of boundary conditions.
\begin{defn}
\label{def:BMN_boundary}Let $\psi=(\psi_{2},\psi_{3})\in\Cl^{1,\alpha}(\overline{\Omega};\R^{2})$
be a $\Cl^{1}$ diffeomorphism of a neighborhood of $\overline{\Omega}$
into $\R^{2}$ such that $J\psi>0$ in $\Omega$. Denote the restriction
of $\psi$ to the boundary $\partial\Omega$ by $\phi=(\phi_{2},\phi_{3}):=\psi_{|\partial\Omega}$.
We say that $\phi_{2},\phi_{3}$ are \emph{BMN boundary conditions}
if $\phi\colon\partial\Omega\to\R^{2}$ maps $\partial\Omega$ onto
a convex closed curve. In this case we write $(\phi_{2},\phi_{3})\in\BMN$.\end{defn}
\begin{rem}
\label{rem:BMN}Note that if $\Omega$ is a convex domain then $\phi_{2}=x_{1}$,
$\phi_{3}=x_{2}$ are BMN boundary conditions.
\end{rem}
The main result in \cite{bauman2001univalent} states that (\ref{eq:complete_ii)})
is satisfied for $k=0$, provided that the boundary conditions are
BMN.
\begin{prop}
[Theorems 2.4-2.5, \cite{bauman2001univalent}] \label{prop:BMN}Suppose
$d=2$. If $(\phi_{2},\phi_{3})\in\BMN$ then
\[
\det\begin{bmatrix}\nabla u_{0}^{2} & \nabla u_{0}^{3}\end{bmatrix}>0\quad\text{in }\Omega.
\]
 If in addition $\Omega$ is a $\Cl^{2}$ domain and $a\in\Czeroone$
then 
\[
\det\begin{bmatrix}\nabla u_{0}^{2} & \nabla u_{0}^{3}\end{bmatrix}>0\quad\text{in }\overline{\Omega}.
\]

\end{prop}
Let us now focus on (\ref{eq:complete_iii)}) for $k=0$. As far as
the author is aware, there is no result concerning the zeros of
\[
\det\begin{bmatrix}u_{0}^{1} & u_{0}^{2} & u_{0}^{3}\\
\nabla u_{0}^{1} & \nabla u_{0}^{2} & \nabla u_{0}^{3}
\end{bmatrix}.
\]
Therefore we need the following statement, whose proof can be found
at the end of this subsection.
\begin{lem}
\label{lem:condition_iii}Suppose $d=2$. Take $\phi_{1},\phi_{2},\phi_{3}\in\gamma\left(\Conealpr\right)$
such that $\phi_{1}$ has a fixed sign and $\left(\phi_{2},\phi_{3}\right),\left(\frac{\phi_{2}}{\phi_{1}},\frac{\phi_{3}}{\phi_{1}}\right)\in\BMN$.
Then 
\[
\det\begin{bmatrix}u_{0}^{1} & u_{0}^{2} & u_{0}^{3}\\
\nabla u_{0}^{1} & \nabla u_{0}^{2} & \nabla u_{0}^{3}
\end{bmatrix}\neq0\quad\text{in \ensuremath{\Omega}.}
\]
If in addition $\Omega$ is a $\Cl^{2}$ domain and $a\in\Czeroone$
then
\[
\det\begin{bmatrix}u_{0}^{1} & u_{0}^{2} & u_{0}^{3}\\
\nabla u_{0}^{1} & \nabla u_{0}^{2} & \nabla u_{0}^{3}
\end{bmatrix}\neq0\quad\text{in \ensuremath{\overline{\Omega}}.}
\]

\end{lem}
We are now in a position to state the main results of this subsection,
concerning the construction of complete sets of measurements in dimension
$d=2$.
\begin{thm}
\label{thm:complete_2d}Suppose $d=2$. Take $\Omega'\Subset\Omega$
and $\phi_{1},\phi_{2},\phi_{3}\in\gamma\left(\Conealpr\right)$ such
that $\phi_{1}$ has a fixed sign and $\left(\phi_{2},\phi_{3}\right),\left(\frac{\phi_{2}}{\phi_{1}},\frac{\phi_{3}}{\phi_{1}}\right)\in\BMN$.
We can choose a finite $K\subseteq\Kad\setminus\Sigma$ such that
\[
K\times\{\phi_{1},\phi_{2},\phi_{3}\}
\]
is a complete set of measurements in $\Omega'$.

If in addition $\Omega$ is a $\Cl^{2}$ domain and $a\in\Czeroone$
then we can choose a finite $K'\subseteq\Kad\setminus\Sigma$ such
that
\[
K'\times\{\phi_{1},\phi_{2},\phi_{3}\}
\]
is a complete set of measurements in $\Omega$.
\end{thm}

\begin{rem}
Note that the hypotheses on the boundary conditions  are satisfied
if $\phi_{1}=1$ and $\left(\phi_{2},\phi_{3}\right)\in\BMN$. Therefore,
if $\Omega$ is convex, by Remark \ref{rem:BMN} an easy choice for
the boundary conditions is $\phi_{1}=1$, $\phi_{2}=x_{1}$ and $\phi_{3}=x_{2}$.\end{rem}
\begin{proof}
By Lemma \ref{lem:condition_iii}, Proposition \ref{prop:BMN} and
Strong Maximum Principle we obtain that 
\[
u_{0}^{1},\det\begin{bmatrix}\nabla u_{0}^{2} & \nabla u_{0}^{3}\end{bmatrix},\det\begin{bmatrix}u_{0}^{1} & u_{0}^{2} & u_{0}^{3}\\
\nabla u_{0}^{1} & \nabla u_{0}^{2} & \nabla u_{0}^{3}
\end{bmatrix}\neq0\quad\text{in \ensuremath{\Omega},}
\]
respectively. Thus we can apply Lemma \ref{lem:preliminary_analyticity}
with $\Omega'\Subset\Omega$, $b=6$, $k^{x}=0$, $\phi_{4}=\phi_{1}$,
$\phi_{5}=\phi_{2}$, $\phi_{6}=\phi_{3}$ and $\zeta(u,v,w,t,z,y)=u\det\begin{bmatrix}\nabla v & \nabla w\end{bmatrix}\det\begin{bmatrix}t & z & y\\
\nabla t & \nabla z & \nabla y
\end{bmatrix}$ to obtain the existence of a finite set $K\subseteq\Kad\setminus\Sigma$
such that 
\[
p(x):=\sum_{k\in K}\Bigl|u_{k}^{1}\det\begin{bmatrix}\nabla u_{k}^{2} & \nabla u_{k}^{3}\end{bmatrix}\det\begin{bmatrix}u_{k}^{1} & u_{k}^{2} & u_{k}^{3}\\
\nabla u_{k}^{1} & \nabla u_{k}^{2} & \nabla u_{k}^{3}
\end{bmatrix}\Bigr|(x)>0,\qquad x\in\overline{\Omega'}.
\]
In the sequel, we shall denote several positive constants independent
of $x\in\overline{\Omega'}$ by $c$. Since $u_{k}^{\phi_{i}}\in\Coner$
(Proposition \ref{pro:holder}), $p(x)\ge c$ for every $x\in\overline{\Omega'}$.
As a result, for any $x\in\Omega'$ there exists $k\in K$ such that
\[
\Bigl|u_{k}^{1}\det\begin{bmatrix}\nabla u_{k}^{2} & \nabla u_{k}^{3}\end{bmatrix}\det\begin{bmatrix}u_{k}^{1} & u_{k}^{2} & u_{k}^{3}\\
\nabla u_{k}^{1} & \nabla u_{k}^{2} & \nabla u_{k}^{3}
\end{bmatrix}\Bigr|(x)\ge c.
\]
As a consequence we have
\[
\left|u_{k}^{1}(x)\right|\ge c,\quad\left|\det\begin{bmatrix}\nabla u_{k}^{2} & \nabla u_{k}^{3}\end{bmatrix}(x)\right|\ge c,\quad\Bigl|\det\begin{bmatrix}u_{k}^{1} & u_{k}^{2} & u_{k}^{3}\\
\nabla u_{k}^{1} & \nabla u_{k}^{2} & \nabla u_{k}^{3}
\end{bmatrix}(x)\Bigr|\ge c,
\]
whence the first part of the theorem follows.

If in addition we suppose that $\Omega$ is a $\Cl^{2}$ domain and
that $a\in\Czeroone$, we can use the second part of Proposition \ref{prop:BMN}
and the second part of Lemma \ref{lem:condition_iii} to infer that
\[
\det\begin{bmatrix}u_{0}^{1} & u_{0}^{2} & u_{0}^{3}\\
\nabla u_{0}^{1} & \nabla u_{0}^{2} & \nabla u_{0}^{3}
\end{bmatrix},\det\begin{bmatrix}\nabla u_{0}^{2} & \nabla u_{0}^{3}\end{bmatrix}\neq0\quad\text{in \ensuremath{\overline{\Omega}},}
\]
respectively. Therefore, arguing as before, we obtain a complete set
in $\Omega$.
\end{proof}
From the proof of Theorem \ref{thm:complete_2d} it is clear that
the assumptions $\phi_{1}\gtrless0$ and $(\phi_{2},\phi_{3})\in\BMN$
allow us to deduce (\ref{eq:complete_i)}) and (\ref{eq:complete_ii)})
of Definition \ref{def:complete sets}, respectively. One may wonder
if the above Theorem still holds without one of these two hypotheses.
The answer is no, as the following examples show. Therefore, if the
boundary conditions $\phi_{i}$ are not chosen properly, in general
one cannot expect to obtain (\ref{eq:complete_i)}) and (\ref{eq:complete_ii)})
by doing many measurements with different wavenumbers but with fixed
illuminations. First, we provide a counterexample for condition (\ref{eq:complete_i)}).
\begin{example}
\label{exa:condition_i}Suppose $d=2$, $\Omega=B(0,1)$ and $a=q=1$.
We choose the boundary condition $\phi_{1}(x_{1},x_{2})=x_{1}$, that
clearly does not have a constant sign. In this case the Helmholtz
equation (\ref{eq:hel_repeated}) can be written in polar coordinates
$(r,\theta)$ and reads
\[
u_{rr}+\frac{u_{r}}{r}+\frac{u_{\theta\theta}}{r^{2}}+ku=0.
\]
It is straightforward to see that
\[
u_{k}^{1}(r,\theta)=\frac{J_{1}(\sqrt{k}r)}{J_{1}(\sqrt{k})}\cos\theta
\]
satisfies the above equation, where $J_{1}$ is the Bessel function
of the first kind of order $1$ and $k>0$ is not an eigenvalue of
the problem. Thus, $\left\{ u_{k}^{1}\right\} $ represents a family
of solution to the Helmholtz equation with fixed boundary condition
and varying wavenumber. However, condition (\ref{eq:complete_i)})
cannot hold since $u_{k}^{1}(0,x_{2})=0$ for every $k$ and for every
$x_{2}\in\left(-1,1\right)$.
\end{example}
Next, we study condition (\ref{eq:complete_ii)}). Since it expresses
the linear independence of the gradient of the solutions inside the
domain, we shall see that it is not possible to require $\phi_{1},\phi_{2}$
to be just linearly independent.
\begin{example}
\label{exa:condition_ii}We consider the situation of Example \ref{exa:condition_i}.
Suppose $\Omega=B(0,1)$ and $a=q=1$. We choose the boundary conditions
$\phi_{2}(x_{1},x_{2})=x_{1}$ and $\phi_{3}=1$. Clearly, $(\phi_{2},\phi_{3})\notin\BMN$
since $\phi_{3}$ is a constant, but $\phi_{2}$, $\phi_{3}$ are
linearly independent. It is straightforward to see that the corresponding
solutions to (\ref{eq:hel_repeated}) are
\[
u_{k}^{2}(r,\theta)=\frac{J_{1}(\sqrt{k}r)}{J_{1}(\sqrt{k})}\cos\theta,\qquad u_{k}^{3}(r,\theta)=\frac{J_{0}(\sqrt{k}r)}{J_{0}(\sqrt{k})},
\]
where $J_{n}$ is the Bessel function of the first kind of order $n$
and $k>0$ is not an eigenvalue of the problem. 

Take a matrix-valued function $A\colon\Omega\to GL(2)$, where $GL(2)$
denotes the set of $2\times2$ invertible matrices. By viewing $A(x)$
as a change of coordinates in $T_{x}\Omega$, the tangent space in
$x$ to $\Omega$, we get 
\[
\det\begin{bmatrix}A\nabla u_{k}^{2}A^{-1} & A\nabla u_{k}^{3}A^{-1}\end{bmatrix}=\det\left(A\begin{bmatrix}\nabla u_{k}^{2} & \nabla u_{k}^{3}\end{bmatrix}A^{-1}\right)=\det\begin{bmatrix}\nabla u_{k}^{2} & \nabla u_{k}^{3}\end{bmatrix}.
\]
Therefore, as far as $\det\begin{bmatrix}\nabla u_{k}^{2} & \nabla u_{k}^{3}\end{bmatrix}$
is concerned, we can express the gradient in any system of coordinates.

In this case, writing $\nabla u_{k}^{i}$ with respect to $e_{\theta}$
and $e_{r}$ we have $\nabla u_{k}^{i}=\frac{1}{r}\frac{\partial u_{k}^{i}}{\partial\theta}e_{\theta}+\frac{\partial u_{k}^{i}}{\partial r}e_{r}$.
Hence
\[
\det\begin{bmatrix}\nabla u_{k}^{2} & \nabla u_{k}^{3}\end{bmatrix}=\det\begin{bmatrix}\frac{1}{r}\frac{\partial u_{k}^{2}}{\partial\theta} & \frac{1}{r}\frac{\partial u_{k}^{3}}{\partial\theta}\\
\cdot & \cdot
\end{bmatrix}=\det\begin{bmatrix}-\frac{1}{r}\frac{J_{1}(\sqrt{k}r)}{J_{1}(\sqrt{k})}\sin\theta & 0\\
\cdot & \cdot
\end{bmatrix}.
\]
Thus, we have $\det\begin{bmatrix}\nabla u_{k}^{2} & \nabla u_{k}^{3}\end{bmatrix}(x_{1},0)=0$
for every $k$ and for every $x_{1}\in\left(-1,1\right)$ and so (\ref{eq:complete_ii)})
cannot be satisfied by using many measurements with these fixed illuminations
and varying wavenumbers.
\end{example}
We end this subsection with the proof of Lemma \ref{lem:condition_iii}. 
\begin{proof}
[Proof. (Lemma \ref{lem:condition_iii})]We shall prove only the first
part of the lemma. The second part can be proved following the same
argument.

For simplicity of notation we write $u_{i}:=u_{0}^{i}$ and suppose
$\phi_{1}>0$. By Proposition \ref{prop:BMN} we have $\det\begin{bmatrix}\nabla u_{2} & \nabla u_{3}\end{bmatrix}>0$
in $\Omega$, and so $\left\{ \nabla u_{2},\nabla u_{3}\right\} $
form a basis of $\R^{2}$ in every point of $\Omega$. Therefore there
exist $\lambda,\mu\colon\Omega\to\R$ such that 
\begin{equation}
\nabla u_{1}=\lambda\nabla u_{2}+\mu\nabla u_{3}\quad\text{ in \ensuremath{\Omega},}\label{eq:cond_iii_1}
\end{equation}
whence
\[
\det\begin{bmatrix}u_{1} & u_{2} & u_{3}\\
\nabla u_{1} & \nabla u_{2} & \nabla u_{3}
\end{bmatrix}=\left(u_{1}-\lambda u_{2}-\mu u_{3}\right)\det\begin{bmatrix}\nabla u_{2} & \nabla u_{3}\end{bmatrix}\quad\text{in }\Omega.
\]

By contradiction, suppose that there exists $x_{0}\in\Omega$ such
that 
\begin{equation}
u_{1}(x_{0})=\lambda(x_{0})u_{2}(x_{0})+\mu(x_{0})u_{3}(x_{0}).\label{eq:cond_iii_1bis}
\end{equation}
Denote $\lambda_{0}=\lambda(x_{0})$, $\mu_{0}=\mu(x_{0})$ and $z=\lambda_{0}u_{2}+\mu_{0}u_{3}$.
Since $\phi_{1}>0$, by the Strong Maximum Principle we get $u_{1}>0$
in $\overline{\Omega}$. Let $h=z/u_{1}$ and $\psi_{i}=u_{i}/u_{1}$
for $i=2,3$. For every $v\in\mathcal{D}(\Omega)$ and $i=2,3$ a
straightforward calculation shows that (see Lemma \ref{lem:trace-uv})
\[
\int_{\Omega}u_{1}^{2}\, a\nabla\psi_{i}\cdot\nabla v\, dx=0.
\]
Therefore $\psi_{i}$ is the unique solution to the problem
\[
\left\{ \begin{array}{l}
-\div(u_{1}^{2}\, a\nabla\psi_{i})=0\qquad\text{in \ensuremath{\Omega},}\\
\psi_{i}=\phi_{i}/\phi_{1}\qquad\text{on \ensuremath{\partial\Omega},}
\end{array}\right.
\]
and by using Proposition \ref{prop:BMN} we infer that
\begin{equation}
\det\begin{bmatrix}\nabla\psi_{2} & \nabla\psi_{3}\end{bmatrix}>0\quad\text{in \ensuremath{\Omega}},\label{eq:cond_iii_2}
\end{equation}
since $\left(\frac{\phi_{2}}{\phi_{1}},\frac{\phi_{3}}{\phi_{1}}\right)\in\BMN$.
By (\ref{eq:cond_iii_1}) and (\ref{eq:cond_iii_1bis}) we have $\nabla u_{1}(x_{0})=\nabla z(x_{0})$
and $u_{1}(x_{0})=z(x_{0})>0$, whence $\nabla(\log u_{1})(x_{0})=\nabla(\log z)(x_{0})$
and so $\nabla(\log h)(x_{0})=0$. As a consequence, $0=\nabla h(x_{0})=\lambda_{0}\nabla\psi_{2}(x_{0})+\mu_{0}\nabla\psi_{3}(x_{0}),$
which contradicts (\ref{eq:cond_iii_2}) since $(\lambda_{0},\mu_{0})\neq(0,0)$.
\end{proof}

\subsection{Complete Sets: Three-Dimensional Case\label{sub:Complete-Sets:-3D}}

Throughout this subsection we assume $d=3$. First, we need to consider
(\ref{eq:complete_ii)}) and (\ref{eq:complete_iii)}) for $k=0$.
In contrast to the case $d=2$, we cannot use Proposition \ref{prop:BMN},
since the equivalent statement is known not to hold in three dimensions
\cite{MR2073507}.

Thus, we assume that the parameter $a$ is a small perturbation of
a constant, symmetric and positive definite matrix $a_{0}$, i.e.
\[
a=a_{0}+s,\qquad\left\Vert s\right\Vert _{\Calpr}\le\delta,
\]
with $\delta$ small enough. The study of this approximation is common
in the literature \cite{triki2010,bal2011quantitative,ammari2012acousto,ammari2012quantitative}.
Note that we do not make any assumptions on $q$.

In order to understand why this approximation is useful, let us consider
the constant case, i.e. $a=a_{0}$. Choose $\phi_{1}=1$, $\phi_{2}=x_{1}$,
$\phi_{3}=x_{2}$ and $\phi_{4}=x_{3}$. Hence $1$, $x_{1}$, $x_{2}$
and $x_{3}$ are the solutions to the problems
\[
\left\{ \begin{array}{l}
-\div(a_{0}\,\nabla u)=0\qquad\text{in \ensuremath{\Omega},}\\
u=\phi_{i}\qquad\text{on \ensuremath{\partial\Omega},}
\end{array}\right.
\]
for $i=1,\dots,4$, respectively. Therefore conditions (\ref{eq:complete_i)}),
(\ref{eq:complete_ii)}) and (\ref{eq:complete_iii)}) are trivially
satisfied in the case $k=0$. Thanks to the multi-frequency approach
we can extend this property to any range of frequencies, and a continuity
argument allows small variations around a constant value. These two
steps are carried out in the following theorem, which concerns the
construction of complete sets of measurements in dimension $d=3$.
\begin{thm}
\label{thm:complete_3D}Suppose $d=3$. Let $a_{0}$ be a constant,
symmetric and positive definite matrix. There exists $\delta>0$ such
that for any $s\in\Calpr$ with $\left\Vert s\right\Vert _{\Calpr}\le\delta$
we can choose a finite $K\subseteq\Kad\setminus\Sigma$ such that
\[
K\times\{1,x_{1},x_{2},x_{3}\}
\]
is a complete set of measurements in $\Omega$ for $a=a_{0}+s$.\end{thm}
\begin{proof}
Fix $0<\epsilon<\mina/3$, where $\mina$ is the smallest eigenvalue
of $a_{0}$. In the following we will consider positive constants
$c$ depending on $a_{0}$, $\Omega$ and $\epsilon$. Take $0<\delta\le\epsilon$
and $s\in\Calpr$ with $\left\Vert s\right\Vert _{\Calpr}\le\delta$.
For $a_{s}=a_{0}+s$ consider $u_{0}^{i}(s)$ the solution to the
problem
\[
\left\{ \begin{array}{l}
-\div(a_{s}\,\nabla u_{0}^{i}(s))=0\qquad\text{in \ensuremath{\Omega},}\\
u_{0}^{i}(s)=\phi_{i}\qquad\text{on \ensuremath{\partial\Omega},}
\end{array}\right.
\]
where $\phi_{1}=1$ and $\phi_{i}=x_{i-1}$ for $i=2,3,4$. Thus,
$\left\Vert u_{0}^{i}(s)\right\Vert _{2}\le c$ for some $c>0$. As
a result, from usual elliptic regularity theory (e.g. (8.38) in \cite{gilbarg2001elliptic})
we obtain $\left\Vert u_{0}^{i}(s)\right\Vert _{\infty}\le c$. By
Theorem 8.33 in \cite{gilbarg2001elliptic} this implies that 
\begin{equation}
\left\Vert \nabla u_{0}^{i}(s)\right\Vert _{\Calpvect}\le c.\label{eq:a_const_1}
\end{equation}

The same argument applied to $u_{0}^{i}(s)-u_{0}^{i}(0)$ as a solution
to 
\[
-\div\left(a_{0}\nabla(u_{0}^{i}(s)-u_{0}^{i}(0))\right)=\div(s\nabla u_{0}^{i}(s)).
\]
gives $\left\Vert u_{0}^{i}(s)-u_{0}^{i}(0)\right\Vert _{\infty}\le c\,\delta$
for some $c>0$. Hence, by (\ref{eq:a_const_1}) and Theorem 8.33
in \cite{gilbarg2001elliptic} applied to $u_{0}^{i}(s)-u_{0}^{i}(0)$
there holds
\[
\left\Vert \nabla u_{0}^{i}(s)-\nabla u_{0}^{i}(0)\right\Vert _{\infty}\le c\left(\left\Vert u_{0}^{i}(s)-u_{0}^{i}(0)\right\Vert _{\infty}+\left\Vert s\nabla u_{0}^{i}(s)\right\Vert _{\Calpvect}\right)\le c\,\delta,
\]
whence
\[
\left\Vert u_{0}^{i}(s)-u_{0}^{i}(0)\right\Vert _{\Coner}\le c\,\delta.
\]
As a consequence, since for $s=0$ we have
\[
u_{0}^{1}(0)=\det\begin{bmatrix}\nabla u_{0}^{2}(0) & \cdots & \nabla u_{0}^{4}(0)\end{bmatrix}=\det\begin{bmatrix}u_{0}^{1}(0) & \cdots & u_{0}^{4}(0)\\
\nabla u_{0}^{1}(0) & \cdots & \nabla u_{0}^{4}(0)
\end{bmatrix}=1\quad\text{in }\overline{\Omega},
\]
there exists $\delta\le\epsilon$ such that
\[
u_{0}^{1}(s),\det\begin{bmatrix}\nabla u_{0}^{2}(s) & \cdots & \nabla u_{0}^{4}(s)\end{bmatrix},\det\begin{bmatrix}u_{0}^{1}(s) & \cdots & u_{0}^{4}(s)\\
\nabla u_{0}^{1}(s) & \cdots & \nabla u_{0}^{4}(s)
\end{bmatrix}\neq0\quad\text{in }\overline{\Omega},
\]
for all $s\in\Calpr$ with $\left\Vert s\right\Vert _{\Calpr}\le\delta$.
Finally, arguing as in the proof of Theorem \ref{thm:complete_2d},
the result follows by using Lemma \ref{lem:preliminary_analyticity}.
\end{proof}

\subsection{Proper Sets of Measurements}

This subsection is devoted to the study of proper sets of measurements,
which were introduced in Subsection \ref{sub:Applications-to-Microwave}.
Since the conditions characterizing proper sets are weaker than those
of complete sets, we are not going into the details: the reader is
referred to the previous parts of this section. We first recall the
definition of proper sets of measurements for $a\in\Calpr$ and $q\in\linfr$
satisfying (\ref{eq:assumption_a}) and (\ref{eq:assumption_q}).

\begin{defproper}
Let $p$ and $r$ be two positive constants. A set of measurements $K\times\left\{ \phi_{i}:i=1,2,3\right\} $ is \emph{proper} in $\Omega'$ if for every $x\in\Omega'$ there exists $\bar{k}=\bar{k}(x)\in K$ such that: \begin{equation} \left|u_{\bar{k}}^{1}(x)\right|\ge p,\tag{PSM 1} \end{equation} \begin{equation} \bigl|\nabla u_{\bar{k}}^{2}(x)\bigr|\bigl|\nabla u_{\bar{k}}^{3}(x)\bigr|\left|\sin\theta_{\nabla u_{\bar{k}}^{2},\nabla u_{\bar{k}}^{3}}(x)\right|\ge r.\tag{PSM 2} \end{equation} The collection of all proper sets of measurements in $\Omega'$ with constants $p$ and $r$ will be denoted by $\mathcal{P}(\Omega';p,r)$.
\end{defproper}The multi-frequency approach discussed in the last section is easily
applicable for the construction of proper sets of measurements. In
general, the choice for the illuminations $1,x_{1}$ and $x_{2}$
gives a proper set in any dimension (in dimension three, $a$ is required
to be close to a constant).

Note that, in general, condition (\ref{eq:proper_ii)}) is weaker
than (\ref{eq:complete_ii)}), and they are equivalent when $d=2$
since 
\begin{equation}
\begin{split}\left|\det\left(\begin{bmatrix}w & z\end{bmatrix}\right)\right| & =\left|w\right|\left|z\right|\left|\sin\theta_{w,z}\right|,\quad w,z\in\R^{2}.\end{split}
\label{eq:link_determ_sine-1}
\end{equation}
Thus, when $d=2$ one can directly  use an analogue version of Theorem
\ref{thm:complete_2d}, without the assumption $(\frac{\phi_{2}}{\phi_{1}},\frac{\phi_{3}}{\phi_{1}})\in\BMN$,
which was needed to satisfy (\ref{eq:complete_iii)}).
\begin{thm}
\label{thm:proper_2d}Suppose $d=2$. Take $\Omega'\Subset\Omega$
and $\phi_{1},\phi_{2},\phi_{3}\in\gamma\left(\Conealpr\right)$ such
that $\phi_{1}$ has a fixed sign and $\left(\phi_{2},\phi_{3}\right)\in\BMN$.
We can choose a finite $K\subseteq\Kad\setminus\Sigma$ such that
\[
K\times\{\phi_{1},\phi_{2},\phi_{3}\}
\]
is a proper set of measurements in $\Omega'$.

If in addition $\Omega$ is a $\Cl^{2}$ domain and $a\in\Czeroone$
then we can choose a finite $K'\subseteq\Kad\setminus\Sigma$ such
that
\[
K'\times\{\phi_{1},\phi_{2},\phi_{3}\}
\]
is a proper set of measurements in $\Omega$.
\end{thm}
In dimension three, it is straightforward to check that the analogue
version of Theorem \ref{thm:complete_3D} is applicable, without the
illumination $x_{3}$.
\begin{thm}
\label{thm:proper_3D}Suppose $d=3$. Let $a_{0}$ be a constant,
symmetric and positive definite matrix. There exists $\delta>0$ such
that for any $s\in\Calpr$ with $\left\Vert s\right\Vert _{\Calpr}\le\delta$
we can choose a finite $K\subseteq\Kad\setminus\Sigma$ such that
\[
K\times\{1,x_{1},x_{2}\}
\]
is a proper set of measurements in $\Omega$ for $a=a_{0}+s$.
\end{thm}

\subsection{Numerical Experiments on the Number of Needed Frequencies}

The theory developed so far gives conditions on the illuminations
to construct complete and proper sets of measurements, provided a
finite number of frequencies are chosen in a fixed range. However,
we do not provide an estimate on the number of the required frequencies,
i.e. the cardinality $\#K$ of $K$. In order to document this point
and see what we can expect, we have performed a numerical test in
$3^{8}=6561$ different cases, namely for different choices of the
parameter distributions in the material. For simplicity, we have decided
to consider the two-dimensional case and the concept of proper sets
of measurements.

Take $\Omega=B(0,1)$. We now describe the coefficients we have used.
Set $c=0.35$, $r=0.2$ and $P_{1}=(-c,-c)$, $P_{2}=(-c,c)$, $P_{3}=(c,-c)$
and $P_{4}=(c,c)$. Let $\chi_{i}$ be the characteristic function
of $B(P_{i},r)$ (or, more precisely, a smooth approximation of it).
Then we set
\[
a=1+\sum_{i=1}^{4}\alpha_{i}\chi_{i},\qquad q=1+\sum_{i=1}^{4}\beta_{i}\chi_{i},
\]
where $\alpha_{i},\beta_{i}\in\{0,1,2\}$ (see Figure \ref{fig:structure}).
This construction gives $3^{8}$ different combinations. 

\begin{figure}
\caption{A particular combination of the coefficients.\label{fig:structure}}
\subfloat{\includegraphics{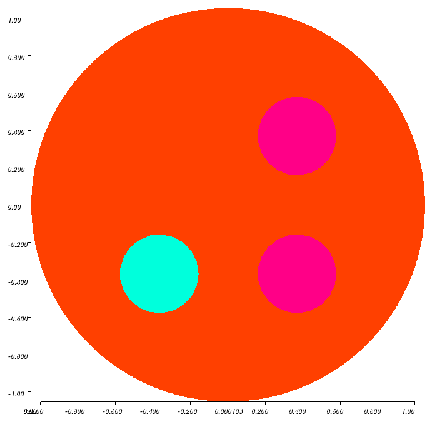}}\qquad{}\qquad{}\qquad{}\subfloat{\includegraphics{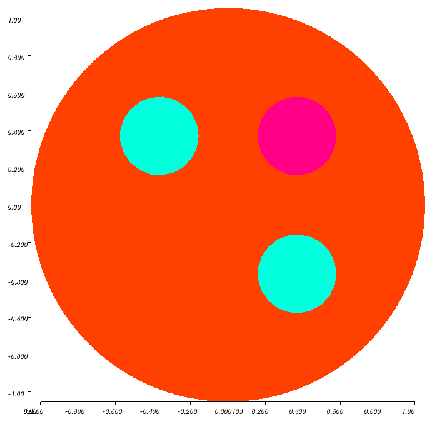}}
\end{figure}

In view of Theorem \ref{thm:proper_2d}, we choose the illuminations
$\phi_{1}=1$, $\phi_{2}=x_{1}$ and $\phi_{3}=x_{2}$. The sequence
$(k_{l})$ introduced in Lemma \ref{lem:preliminary_analyticity}
is given by $k_{l}=\lambda_{0}+a+\frac{b}{l+1}$, for some $a,b>0$
such that $\lambda_{0}+a+b<\lambda_{1}$. (As in Proposition \ref{pro:well-posed},
the sequence $(\lambda_{i})$ denotes the Dirichlet eigenvalues of
the problem.) These conditions ensure that the chosen frequencies
lie between the first and the second Dirichlet eigenvalue, which simplifies
the numerical computation. For any combination of the coefficients,
we compute the minimum number $\#K$  of needed frequencies such that
$\{k_{l}:l=0,\dots,\#K-1\}\times\{\phi_{i}\}$ is a proper set of
measurements in $\Omega$. The results are summarized in Table~\ref{tab:howmany}.
These figures suggest that in practical applications the number of
needed frequencies could be quite small. 
\begin{table}[h]
\caption{Number of combinations of coefficients per number of needed frequencies
to obtain a proper set of measurements in $\Omega$.\label{tab:howmany}}
\begin{tabular}{>{\centering}p{4.5cm}|>{\centering}p{1.4cm}>{\centering}p{1.4cm}>{\centering}p{1.4cm}}
Needed frequencies ($\#K$) & $2$ & $3$ & $\ge4$\tabularnewline
\hline 
Combinations of coefficients & 1609 & 4952 & 0\tabularnewline
\end{tabular}
\end{table}

\section{\label{sec:Microwave-Imaging-by}Applications to Microwave Imaging
by Ultrasound Deformation}

This section is devoted to the discussion of the hybrid problem introduced
in Subsection \ref{sub:Applications-to-Microwave} and to the proofs
of the results therein presented.

\subsection{\label{sub:Formulation-of-the}Formulation of the Problem}

The object under examination is a smooth bounded domain $\Omega\subseteq\R^{d}$,
for $d=2$ or $d=3$. In the microwave regime, the electric field
$u_{k}^{\phi}$ in $\Omega$ is assumed to satisfy the Helmholtz equation

\begin{equation}
\left\{ \begin{array}{l}
-\div(a\,\nabla u_{k}^{\phi})-k\, q\, u_{k}^{\phi}=0\qquad\text{in \ensuremath{\Omega},}\\
u_{k}^{\phi}=\phi\qquad\text{on \ensuremath{\partial\Omega},}
\end{array}\right.\label{eq:physical_modeling-2}
\end{equation}
where $a\in\Calprsca$ is the inverse of the magnetic permeability,
$q\in\linfr$ is the electric permittivity and $\sqrt{k}$ is the
wavenumber (with an abuse of language we shall sometimes refer to
$k$ as the frequency or the wavenumber). We assume (\ref{eq:assumption_a})
and (\ref{eq:assumption_q}). The use of the Helmholtz equation is
a very common scalar approximation of the Maxwell equations \cite{bulyshev2000three,vogelius2000asymptotic,cap2011,bal2012_review}.
As we have seen in Section \ref{sec:Helmholtz_Equation}, problem
(\ref{eq:physical_modeling-2}) is well-posed provided that $k$ is
not a resonant frequency. 

Practitioners can choose a frequency $k>0$ in a fixed range (not
a resonant frequency), a real illumination $\phi$ on the boundary
and measure the generated current on the boundary $a\frac{\partial u_{k}^{\phi}}{\partial\nu}$.
As described in \cite{cap2011}, these measurements are combined with
localized ultrasonic waves focusing in small regions $\omega$ inside
$\Omega$. We assume that the electromagnetic parameters are affected
linearly with respect to the amplitude of the ultrasonic perturbation,
that is supposed to be small. Moreover, this modification is localized
only in the region $\omega$. Under these assumptions, denoting the
modified coefficients by $\bar{a}$ and $\bar{q}$ we have

\[
\left\{ \begin{array}{c}
\bar{a}=a\,\left(1+c_{a}\alpha\chi_{\omega}\right)\text{,}\\
\bar{q}=q\,\left(1+c_{q}\alpha\chi_{\omega}\right)\text{,}
\end{array}\right.
\]
where $\alpha$ is the amplitude of the ultrasonic signal and $c_{a}$
and $c_{q}$ are known functions. The corresponding electric field
$\bar{u}_{k}^{\phi}$ is the solution to
\[
\left\{ \begin{array}{l}
-\div(\bar{a}\,\nabla\bar{u}_{k}^{\phi})-k\,\bar{q}\,\bar{u}_{k}^{\phi}=0\qquad\text{in \ensuremath{\Omega},}\\
\bar{u}_{k}^{\phi}=\phi\qquad\text{on \ensuremath{\partial\Omega}}.
\end{array}\right.
\]
The density current $a\frac{\partial\bar{u}_{k}^{\phi}}{\partial\nu}$
on the boundary of the domain is a measurable datum. 

We now see how the internal energies can be determined by studying
the change between $a\frac{\partial u_{k}^{\phi}}{\partial\nu}$ and
$a\frac{\partial\bar{u}_{k}^{\phi}}{\partial\nu}$. We consider a
wavenumber $k$ and two fixed illuminations $\phi$ and $\psi$. Suppose
the domain $\omega$ to be a small ball inside $\Omega$ of centre
$z$. By asymptotic analysis techniques \cite{ammari2007polarization,cap2011},
there holds
\begin{multline}
\int_{\partial\Omega}a\left(\frac{\partial\bar{u}_{k}^{\phi}}{\partial\nu}-\frac{\partial u_{k}^{\phi}}{\partial\nu}\right)\psi\, d\sigma\\
=\left|\omega\right|\frac{d\, c_{a}(z)\alpha}{c_{a}(z)\alpha+d}\, a(z)\nabla u_{k}^{\phi}(z)\cdot\nabla u_{k}^{\psi}(z)+\left|\omega\right|k\, c_{q}(z)\alpha\, q(z)u_{k}^{\phi}(z)u_{k}^{\psi}(z)+o(\left|\omega\right|),\label{eq:obtain_internal energiese}
\end{multline}
where $\left|\omega\right|$ denotes the Lebesgue measure of $\omega$
and $d\sigma$ the integration with respect to the surface area. Since
the left hand side is known, by choosing different values for $\alpha$
the internal power density data
\[
E_{k}^{\phi\psi}(z)=a(z)\nabla u_{k}^{\phi}(z)\cdot\nabla u_{k}^{\psi}(z),\qquad e_{k}^{\phi\psi}(z)=q(z)\, u_{k}^{\phi}(z)u_{k}^{\psi}(z)
\]
are recovered for every $z\in\Omega'$, where $\Omega'\Subset\Omega$
is the set of all possible centres where the ultrasonic beams are
focused. Note that the ``polarized'' data with $\phi\neq\psi$ is
available for a fixed $k$, but this argument does not allow the reconstruction
of the cross-frequency data
\[
E_{kl}^{\phi\psi}(z)=a(z)\nabla u_{k}^{\phi}(z)\cdot\nabla u_{l}^{\psi}(z),\qquad e_{kl}^{\phi\psi}(z)=q(z)\, u_{k}^{\phi}(z)u_{l}^{\psi}(z).
\]
We thus chose not to use cross-frequency data for the reconstruction,
in contrast to \cite{cap2011}.

Let us now represent the formulation of Problem \ref{prob:Choose-a-suitable}.
Given a set of measurement $K\times\{\phi_{i}\}$ we consider the
unique solution $u_{k}^{i}\in\Honer$ to the problem
\begin{equation}
\left\{ \begin{array}{l}
-\div(a\,\nabla u_{k}^{i})-k\, q\, u_{k}^{i}=0\qquad\text{in \ensuremath{\Omega},}\\
u_{k}^{i}=\phi_{i}\qquad\text{on \ensuremath{\partial\Omega}.}
\end{array}\right.\label{eq:hel}
\end{equation}

We define the internal data by 
\begin{equation}
e_{kl}^{ij}=q\, u_{k}^{i}\, u_{l}^{j},\qquad E_{kl}^{ij}=a\,\nabla u_{k}^{i}\cdot\nabla u_{l}^{j}.\label{eq:internal energies}
\end{equation}
For simplicity, we denote $e_{kl}=(e_{kl}^{ij})_{ij}$ and $e_{k}:=e_{kk}$
and similarly for $E$. The matrices $e_{k}$ and $E_{k}$ are to
be considered as known matrix-valued functions. We study the inverse
problem of determining the parameters $a$ and $q$ in $\Omega'$
from the knowledge in $\Omega'$ of $e_{k}$ and $E_{k}$ with a properly
chosen set of measurements. Note that this reconstruction problem
is slightly different to the one studied in \cite{cap2011}, where
the full matrices 
\[
e=\begin{bmatrix}e_{k_{1}k_{1}} & \cdots & e_{k_{1}k_{M}}\\
\vdots & \ddots & \vdots\\
e_{k_{M}k_{1}} & \cdots & e_{k_{M}k_{M}}
\end{bmatrix},\quad E=\begin{bmatrix}E_{k_{1}k_{1}} & \cdots & E_{k_{1}k_{M}}\\
\vdots & \ddots & \vdots\\
E_{k_{M}k_{1}} & \cdots & E_{k_{M}k_{M}}
\end{bmatrix}.
\]
are supposed to be known, with $K=\{k_{1},\dots,k_{M}\}$. In our
case, we suppose that only the diagonal blocks are measurable.

\subsection{\label{sub:Reconstruction-algorithm}Reconstruction Algorithm}

This subsection is devoted to the proofs of the exact formulae given
in Subsection \ref{sub:Applications-to-Microwave}.

\subsubsection{Reconstruction Formula for $a/q$}

We restate here Theorem \ref{thm:formulaaq}. The new contribution
of this work is the proof of the bounds (\ref{eq:positivity of traces}).

\begin{thmformulaaq}
\label{thm:formulaaq-1} Take $p,r>0$ and $K\times\left\{ \phi_{i}\right\} \in\mathcal{P}(\Omega';p,r)$. Suppose that 
\begin{equation} \left|E_{k}^{ii}(x)\right|,\left|e_{k}^{ii}(x)\right|\le b,\qquad x\in\Omega',\tag{\ref{eq:hp_with_M}}\label{eq:hp_with_M-1} 
\end{equation} 
for some $b>0$. Take $x\in\Omega'$ and $\bar{k}$ as in Definition \ref{def:proper sets}. Then there exists $C=C(\mina,\minq,b)>0$ such that  
\begin{equation} \frac{\tr(e_{\bar{k}})\,\tr(E_{\bar{k}})-\tr(e_{\bar{k}}E_{\bar{k}})}{\tr(e_{\bar{k}})^{2}}(x)\ge C\, p^{2}r^{4},\tag{\ref{eq:positivity of traces}} 
\end{equation} and $a/q$ is given in terms of the data by \begin{equation} \left|\nabla(e_{\bar{k}}/\tr(e_{\bar{k}}))\right|_{2}^{2}\,\frac{a}{q}=2\,\frac{\tr(e_{\bar{k}})\,\tr(E_{\bar{k}})-\tr(e_{\bar{k}}E_{\bar{k}})}{\tr(e_{\bar{k}})^{2}}\quad\text{in }x.\tag{\ref{eq:formula for a/q}}
\end{equation}
\end{thmformulaaq}
\begin{rem}
It is worth noting that in presence of noise the dependence of $\bar{k}$
on $x$ does not constitute a source of instability during\foreignlanguage{english}{
the reconstruction performed in (\ref{eq:formula for a/q}). Take
$0<\epsilon<\min(p,r)$. Since }$K=\{k_{m}:m=1,\dots,M\}$\foreignlanguage{english}{
is finite, we define for each }$m\in\{1,\dots,M\}$\foreignlanguage{english}{
the set $G_{m}=\{x\in\Omega':\left|u_{k_{m}}^{1}(x)\right|>p-\epsilon\text{ and }\bigl|\nabla u_{k_{m}}^{2}(x)\bigr|\bigl|\nabla u_{k_{m}}^{3}(x)\bigr|\bigl|\sin\theta_{\nabla u_{k_{m}}^{2},\nabla u_{k_{m}}^{3}}(x)\bigr|>r-\epsilon\}$.
In view of} Proposition \ref{pro:holder}, $G_{m}$ is open. By Definition
\ref{def:proper sets}, $\Omega'$ can be written as $\Omega'=\cup_{m}G_{m}$,
and so we can consider a smooth partition of unity $\{\psi_{m}\}_{m}$
subject to the cover $\{G_{m}\}_{m}$. Denote by $r_{m}$ the value
for $a/q$ obtained via (\ref{eq:formula for a/q}) with $\bar{k}=k_{m}$.
In view of Theorem \ref{thm:formulaaq}, $r_{m}$ is a well-defined
function in $G_{m}$. We can now reconstruct $a/q$ in $\Omega'$
using
\[
\frac{a}{q}(x)=\sum_{m=1}^{M}r_{m}(x)\psi_{m}(x),\qquad x\in\Omega'.
\]
Thus, with this formulation we have removed the dependence of $\bar{k}$
on $x$ and so no additional instability occurs.\end{rem}
\begin{proof}
Condition (\ref{eq:proper_i)}) implies $\tr(e_{\bar{k}})>0$ in $\Omega'$,
and so we may divide by $\tr(e_{\bar{k}})$. Following the argument
given in the proof of Proposition 3.3 in \cite{cap2011} we obtain
\[
\left|\nabla(e_{\bar{k}}/\tr(e_{\bar{k}}))\right|_{2}^{2}\,\frac{a}{q}=2\,\frac{\tr(e_{\bar{k}})\,\tr(E_{\bar{k}})-\tr(e_{\bar{k}}E_{\bar{k}})}{\tr(e_{\bar{k}})^{2}}.
\]
We now prove (\ref{eq:positivity of traces}). For cleanliness of
notation we shall denote several positive constants depending on $\lambda$
and $b$ simply by $c=c(\lambda,b)>0$. Until the end of the proof,
all the functions will be considered evaluated in $x$. We equip the
space of real symmetric matrices with the Hilbert-Schmidt scalar product
defined by $\langle A,B\rangle=\tr(A\, B)$. We claim that
\begin{equation}
\left|\left(E_{\bar{k}}^{2,3}\right)^{2}-E_{\bar{k}}^{2,2}E_{\bar{k}}^{3,3}\right|\le c\left\Vert E_{\bar{k}}\right\Vert (1+\frac{\tr(E_{\bar{k}})}{\tr(e_{\bar{k}})})\sqrt{1-\cos\theta_{e_{\bar{k}},E_{\bar{k}}}}.\label{eq:1st_ineq}
\end{equation}
As a consequence of this inequality, which we shall prove later, we
get 
\begin{equation}
\sqrt{1-\cos\theta_{e_{\bar{k}},E_{\bar{k}}}}\left\Vert E_{\bar{k}}\right\Vert \ge c\,\frac{r^{2},}{1+\frac{\tr(E_{\bar{k}})}{\tr(e_{\bar{k}})}},\label{eq:3rd_ineq}
\end{equation}
since by (\ref{eq:proper_ii)}) there holds
\[
\left|\left(E_{\bar{k}}^{2,3}\right)^{2}\!-\! E_{\bar{k}}^{2,2}E_{\bar{k}}^{3,3}\right|=a^{2}\left|\left(\nabla u_{\bar{k}}^{2}\cdot\nabla u_{\bar{k}}^{3}\right)^{2}-\left|\nabla u_{\bar{k}}^{2}\right|^{2}\bigl|\nabla u_{\bar{k}}^{3}\bigr|^{2}\right|\ge cr^{2}.
\]
Note that, from the definition of $e_{k}$ and $E_{k}$, we have $\left(e_{\bar{k}}^{ij}\right)^{2}=e_{\bar{k}}^{ii}e_{\bar{k}}^{jj}$
and $\left(E_{\bar{k}}^{ij}\right)^{2}\le E_{\bar{k}}^{ii}E_{\bar{k}}^{jj}$.
In particular,
\begin{equation}
\left\Vert e_{\bar{k}}\right\Vert =\sqrt{\sum_{i,j}\left(e_{\bar{k}}^{ij}\right)^{2}}=\tr(e_{\bar{k}}),\quad\left\Vert E_{\bar{k}}\right\Vert =\sqrt{\sum_{i,j}\left(E_{\bar{k}}^{ij}\right)^{2}}\le\tr(E_{\bar{k}}).\label{eq:4th_ineq}
\end{equation}
Combining (\ref{eq:3rd_ineq}) and (\ref{eq:4th_ineq}) we obtain
\[
\begin{split}\frac{\tr(e_{\bar{k}})\,\tr(E_{\bar{k}})-\tr(e_{\bar{k}}E_{\bar{k}})}{\tr(e_{\bar{k}})^{2}} & \ge\frac{\tr(e_{\bar{k}})(1-\cos\theta_{e_{\bar{k}},E_{\bar{k}}})\left\Vert E_{\bar{k}}\right\Vert }{\tr(e_{\bar{k}})^{2}}\\
 & \ge\frac{c\,\mina^{4}r^{4}\tr(e_{\bar{k}})}{\tr(e_{\bar{k}})^{2}\tr(E_{\bar{k}})(1+\frac{\tr(E_{\bar{k}})}{\tr(e_{\bar{k}})})^{2}},
\end{split}
\]
which gives the desired result since the denominator is bounded by
a constant depending on $b$, and $\tr(e_{\bar{k}})\ge\minq p^{2}$.

Let us now turn to the proof of (\ref{eq:1st_ineq}). We use the notation
$g=E_{\bar{k}}/\left\Vert E_{\bar{k}}\right\Vert -e_{\bar{k}}/\left\Vert e_{\bar{k}}\right\Vert $.
By (\ref{eq:hp_with_M}) and (\ref{eq:4th_ineq}) we have
\[
\begin{split}\!\left|\!\left(E_{\bar{k}}^{2,3}\right)^{2}\!-E_{\bar{k}}^{2,2}E_{\bar{k}}^{3,3}\right| & \le\!\biggl|\left(E_{\bar{k}}^{2,3}\right)^{2}\!-\!\frac{\left\Vert E_{\bar{k}}\right\Vert ^{2}}{\left\Vert e_{\bar{k}}\right\Vert ^{2}}\left(e_{\bar{k}}^{2,3}\right)^{2}\biggr|+\biggl|\frac{\left\Vert E_{\bar{k}}\right\Vert ^{2}}{\left\Vert e_{\bar{k}}\right\Vert ^{2}}e_{\bar{k}}^{2,2}e_{\bar{k}}^{3,3}\!-\!\frac{\left\Vert E_{\bar{k}}\right\Vert }{\left\Vert e_{\bar{k}}\right\Vert }e_{\bar{k}}^{2,2}E_{\bar{k}}^{3,3}\biggr|\\
 & \quad+\left|\frac{\left\Vert E_{\bar{k}}\right\Vert }{\left\Vert e_{\bar{k}}\right\Vert }e_{\bar{k}}^{2,2}E_{\bar{k}}^{3,3}-E_{\bar{k}}^{2,2}E_{\bar{k}}^{3,3}\right|\\
 & =\left\Vert E_{\bar{k}}\right\Vert \biggl|E_{\bar{k}}^{2,3}+\frac{\left\Vert E_{\bar{k}}\right\Vert }{\left\Vert e_{\bar{k}}\right\Vert }e_{\bar{k}}^{2,3}\biggr|\left|g_{2,3}\right|+\frac{\left\Vert E_{\bar{k}}\right\Vert {}^{2}}{\left\Vert e_{\bar{k}}\right\Vert }e_{\bar{k}}^{2,2}\left|g_{3,3}\right|\\
 & \quad+\left\Vert E_{\bar{k}}\right\Vert E_{\bar{k}}^{3,3}\left|g_{2,2}\right|\\
 & \le c\left\Vert E_{\bar{k}}\right\Vert (1+\frac{\tr(E_{\bar{k}})}{\tr(e_{\bar{k}})})\sqrt{1-\cos\theta_{e_{\bar{k}},E_{\bar{k}}}},
\end{split}
\]
since $\left|g_{2,3}\right|+\left|g_{3,3}\right|+\left|g_{2,2}\right|\le c\left\Vert g\right\Vert \le c\sqrt{1-\cos\theta_{e_{\bar{k}},E_{\bar{k}}}}$.
\end{proof}

\subsubsection{\label{subsub:Reconstruction_q}Reconstruction Formula for $q$}

In this subsection we derive the reconstruction formula for $q$ given
in Theorem \ref{thm:formulaq}. It is based on the knowledge of the
ratio $G:=a/q$, that can be computed via the formula (\ref{eq:formula for a/q}).
Since this reconstruction involves the derivative of the data, we
need a stable way to obtain $q$ from $G$.

The following lemma reviews the derivatives and the trace of products
of functions in Sobolev Spaces.
\begin{lem}
\label{lem:trace-uv} Take $u,v\in\Hone\cap\linf$. Then $uv\in\Hone$,
$\nabla(uv)=u\,\nabla v+v\,\nabla u$ and $\gamma(uv)=\gamma(u)\gamma(v)$.
If in addition $v\ge c>0$ almost everywhere then $u/v\in\Hone$,
$\nabla(u/v)=\nabla u/v-u\,\nabla v/v^{2}$ and $\gamma(u/v)=\gamma(u)/\gamma(v)$.\end{lem}
\begin{proof}
It easily follows from \cite[Section 5.5, Theorem 1]{evans} and  \cite[Lemma 7.5]{gilbarg2001elliptic}.
\end{proof}
We need the following preliminary result, which follows from a straightforward
calculation by Lemma \ref{lem:trace-uv} and Proposition \ref{pro:holder}.
The proof is left to the reader.
\begin{lem}
\label{lem:u_i-u_j}Let $K\times\{\phi_{i}\}$ be a set of measurements.
Then $u_{k}^{i}u_{k}^{j}\in\Honer$ is the unique solution to the
problem
\begin{equation}
\left\{ \begin{array}{l}
-\div\left(a\,\nabla(u)\right)=2ke_{k}^{ij}-2E_{k}^{ij}\qquad\text{in \ensuremath{\Omega},}\\
u=\phi_{i}\phi_{j}\qquad\text{on \ensuremath{\partial\Omega},}
\end{array}\right.\label{eq:u_i-u_j}
\end{equation}
for every $k,i,j$.
\end{lem}
We are now in a position to prove the reconstruction formula for $q$
given in Theorem \ref{thm:formulaq}. Recall that if we consider a
proper set of measurements $K\times\{\phi_{i}\}$ we have that
\begin{equation}
\tr(e)=\sum_{k,i}e_{k}^{ii}\ge c>0\quad\text{in \ensuremath{\Omega}}',\label{eq:tr>0}
\end{equation}
for some $c>0$. 

\begin{thmformulaq}
Let \foreignlanguage{english}{\textup{$K\times\left\{ \phi_{i}\right\} $}} be a proper set of measurements. Suppose $q\in\Honer$. Then $\log q$ is the unique solution to the problem \[ \left\{ \begin{array}{l} -\div\left(G\,\tr(e)\,\nabla u\right)=-\div\left(G\,\nabla\left(\tr(e)\right)\right)+2\sum_{k,i}\left(E_{k}^{ii}-k e_{k}^{ii}\right)\quad\text{in }\Omega',\\ u=\log q_{|\partial\Omega'}\qquad\text{on \ensuremath{\partial\Omega}}'. \end{array}\right. \]
\end{thmformulaq}
\begin{proof}
By Proposition \ref{pro:holder} we infer that $u_{k}^{i},u_{k}^{j}\in\Honer\cap\linfr$.
Therefore, by Lemma \ref{lem:trace-uv} we get that $u_{k}^{i}u_{k}^{j}\in\Honer\cap\linfr$.
Since $q\in\Honer\cap\linfr$ we obtain that $e_{k}^{ij}\in\Honer$.
Hence by Lemma 7.5 in \cite{gilbarg2001elliptic} we have
\[
\nabla(u_{k}^{i}u_{k}^{j})=\nabla\left(e_{k}^{ij}/q\right)=(\nabla e_{k}^{ij}-e_{k}^{ij}\nabla(\log q))/q,
\]
with $\log q\in\Honer$. As a result, in view of Lemma \ref{lem:u_i-u_j}
it is immediate to show that for every $v\in\mathcal{D}(\Omega)$
\[
\left(2ke_{k}^{ij}-2E_{k}^{ij}\right)(v)=-\div(G\,\nabla e_{k}^{ij})(v)+\div\left(G\, e_{k}^{ij}\,\nabla(\log q)\right)(v),
\]
whence the result follows by summing all these equations with $i=j$.
\end{proof}

\subsection{\label{sub:Numerical-experiment}Numerical Experiments}

In this section we shall do some numerical simulations of the reconstruction
algorithm discussed in the last section. FreeFem++ has been used.
The exact formulae given in Theorem \ref{thm:formulaaq} and Theorem
\ref{thm:formulaq} will be used to image the electromagnetic parameters
$a$ and $q$.

In both cases, the construction of proper sets of measurements by
means of the multi-frequency approach turns out to be effective. Further,
in two dimensions, the reconstruction procedure gives better results
than the one described in \cite{cap2011}: with about one seventh
of the available data, the reconstruction errors are about half of
the previous ones.

\subsubsection{\label{subsub:Two-Dimensional-Example}Two-Dimensional Example}

Since the two-dimensional case has been tested thoroughly in \cite{cap2011},
we have decided to study the same example in order to be able to make
a comparison. There are two main differences. First, the formula for
the reconstruction of $q$ is not the same. Second, in our case the
available data is smaller since we do not use the cross-frequency
data.

Let $\Omega=B(0,1)$ be the unit disk. We use a uniform mesh of the
disk with about 3000 triangles and 1600 vertices. The coefficients
are given by 
\[
a=\begin{cases}
2 & \text{in }B,\\
1.2 & \text{in }C,\\
2.5 & \text{in }E,\\
1 & \text{otherwise,}
\end{cases}\quad q=\begin{cases}
2 & \text{in }B,\\
1.8 & \text{in }C,\\
1.2 & \text{in }E,\\
1 & \text{otherwise.}
\end{cases}
\]
The set $B$ is the rectangle with diagonal $(0,0.4)-(0.3,0.5)$.
The set $C$ is the area delimited by the curve $t\mapsto(0.3+\rho(t)\cos(t),-0.2+\rho(t)\sin(t))$,
where $\rho(t)=(20+3\sin(5t)-2\sin(15t)+\sin(25t))/100$. The set
$E$ is the ellipse of centre $(-0.3,0.1)$, with vertical axis of
length $0.3$ and horizontal axis of length $0.2$. These parameters
represent three inclusions of different contrast in a homogeneous
background medium. This is the typical practical situation, since
cancerous tissues have typically higher values in the parameters \cite{widlak2012hybrid}.
The coefficients $a$ and $q$ are shown in Figure \ref{fig:reference_2d}.

\begin{figure}
\caption{The reference parameters.}
\subfloat[The coefficient $a$.]{\includegraphics{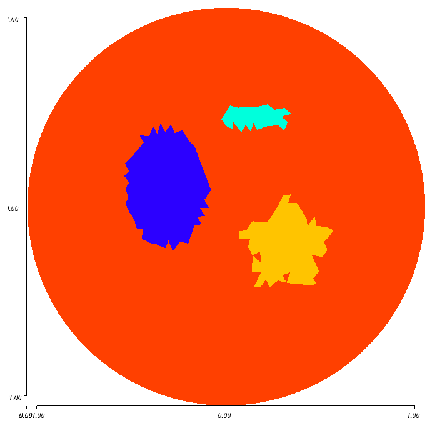}}\qquad{}\qquad{}\qquad{}\label{fig:reference_2d}\subfloat[The coefficient $q$.]{\includegraphics{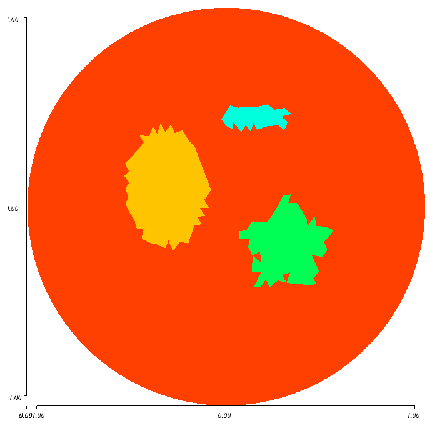}}
\end{figure}

\begin{figure}
\caption{The reconstructed parameters.}
\subfloat[The coefficient $a^{*}$.]{\includegraphics{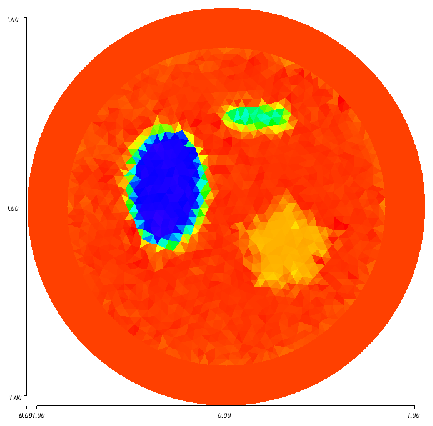}}\qquad{}\qquad{}\qquad{}
\label{fig:reconstr_2d}\subfloat[The coefficient $q^{*}$.]{\includegraphics{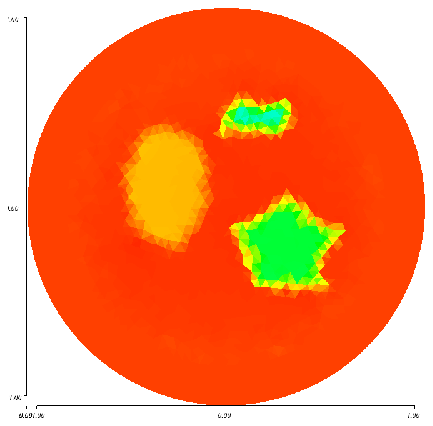}}
\end{figure}

Let $\Omega'=B(0,0.8)$ be the subdomain where the internal energies
are constructed (see Section \ref{sub:Formulation-of-the}). We also
suppose that $a$ and $q$ are known in $\Omega\setminus\Omega'$.
In view of Theorem \ref{thm:proper_2d}, we choose the illuminations
$x_{1}+2$ and $x_{2}+2$ and $K=\{1,3,7\}$. In \cite{cap2011},
the boundary conditions $1,x_{1},x_{2}$ and the same $K$ were chosen:
these illuminations satisfy the hypotheses of Theorem \ref{thm:proper_2d}.

Let $a^{*}$ and $q^{*}$ denote the approximated coefficients. We
first reconstruct $G=a^{*}/q^{*}$ in $x\in\Omega'$ by means of the
formula (\ref{eq:formula for a/q})
\[
\frac{a^{*}}{q^{*}}(x)=2\,\frac{\tr(e_{k})\,\tr(E_{k})-\tr(e_{k}E_{k})}{\left|\nabla(e_{k}/\tr(e_{k}))\right|_{2}^{2}\,\tr(e_{k})^{2}}(x),
\]
averaging over all the $k\in K$ such that the denominator of the
right hand side is bigger than $10^{-2}$. Since for every $x\in\Omega'$
the set of such $k$ is not empty, the chosen set of measurements
turns out to be proper in $\Omega'$.

Then, we use Theorem \ref{thm:formulaq} to image $\log q^{*}\in H_{0}^{1}(\Omega')$:
\[
-\div\left(G\,\tr(e)\,\nabla(\log q^{*})\right)=-\div\left(G\,\nabla\left(\tr(e)\right)\right)+2\sum_{k\in K}\sum_{i=1}^{2}\left(E_{k}^{ii}-k\, e_{k}^{ii}\right)\quad\text{in }\Omega'.
\]
Finally, $a^{*}$ is given by $a^{*}=Gq^{*},$ which, in absence of
noise, gives a good approximation. The reconstructed coefficients
are shown in Figure \ref{fig:reconstr_2d}. 

In Table \ref{tab:Comparison-between-numerical} we compare these
findings with the numerical experiments performed in \cite{cap2011}.
Even if the non-availability of the cross-frequency data makes the
number of measurable internal energies much smaller (about one seventh),
the $L^{2}$ norms of the errors $a-a^{*}$ and $q-q^{*}$ in this
work are about half the corresponding norms obtained in \cite{cap2011}.
\begin{table}
\caption{\label{tab:Comparison-between-numerical}A comparison between the
numerical experiments carried out in \cite{cap2011} and in this work.}
\begin{tabular}{|c|>{\centering}m{3.8cm}|>{\centering}m{4cm}|}
\cline{2-3} 
\multicolumn{1}{c|}{} & Numerical experiments \linebreak in \cite{cap2011} & Numerical experiments\linebreak in this work\tabularnewline
\hline 
Illuminations & $1,x_{1},x_{2}$ & $x_{1}+2,x_{2}+2$\tabularnewline
\hline 
Frequencies $K$ & $1,3,7$ & $1,3,7$\tabularnewline
\hline 
Number of energies & 81 & 12\tabularnewline
\hline 
$\left\Vert a-a^{*}\right\Vert _{2}$ & $3.5\cdot10^{-1}$ & $1.6\cdot10^{-1}$\tabularnewline
\hline 
$\left\Vert q-q^{*}\right\Vert _{2}$ & $1.5\cdot10^{-1}$ & $0.8\cdot10^{-1}$\tabularnewline
\hline 
\end{tabular}
\end{table}

\subsubsection{\label{subsub:Three-Dimensional-Example}Three-Dimensional Example}

Take $\Omega=B(0,1)$ and $\Omega'=B(0,0.8)$. We use a mesh with
about 30000 tetrahedra and 6000 vertices. For simplicity, we choose
$a=1$ and $q=1+0.8\,\chi_{B(0,0.3)}$ and are interested in imaging
the electric permittivity $q$. In view of Theorem \ref{thm:proper_3D},
we choose the set of measurements $\left\{ 1,3,7\right\} \times\{x_{1}+2,x_{2}+2\}$,
which a posteriori turns out to be proper in $\Omega'$. The same
reconstruction procedure described in the two-dimensional example
is used. The reconstruction error is $\left\Vert q-q^{*}\right\Vert _{2}\approx1.3\cdot10^{-1}$.
A comparison between the reference and the reconstructed parameters
is shown in Figure \ref{fig:3D}.
\begin{figure}[t]
\caption{A section of the electric permittivity $q$ in $\Omega'$.}
\subfloat[The reference parameter.]{\includegraphics{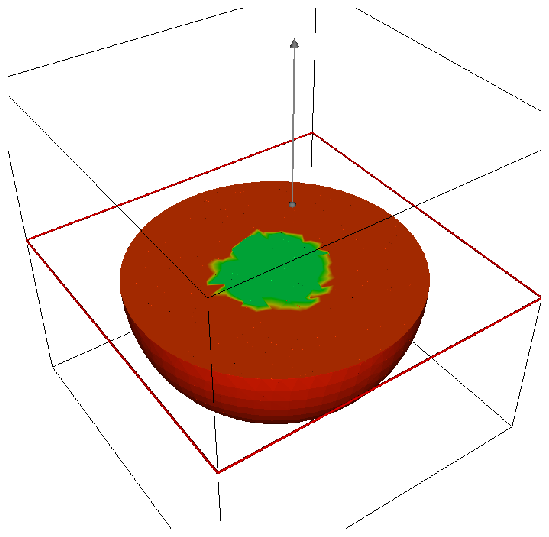}}\qquad{}\qquad{}\label{fig:3D}\subfloat[The reconstructed parameter.]{\includegraphics{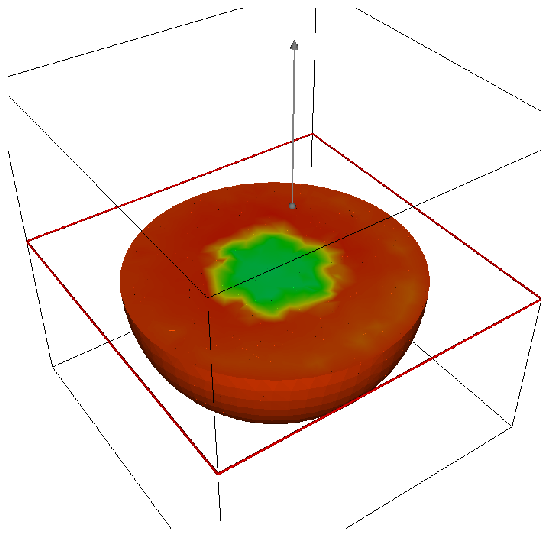}}
\end{figure}

\section*{Acknowledgments}

This work has been done during my D.Phil. at the Oxford Centre for
Nonlinear PDE under the supervision of Yves Capdeboscq, which I would
like to thank for many stimulating discussions on the topic and for
a thorough revision of the paper. I am supported by an EPSRC Research
Studentship and by the EPSRC Science \& Innovation Award to the Oxford
Centre for Nonlinear PDE (EP/EO35027/1). I would like to thank the
referees, whose comments and suggestions greatly improved the quality
of this paper.

\bibliographystyle{plain}
\bibliography{biblio}

\end{document}